\documentclass[11pt]{smfart}
\usepackage{amssymb,latexsym}
\usepackage{graphicx}
\usepackage{version}
\usepackage[all]{xy} 
\usepackage{lscape}
\entrymodifiers={!!<0pt,0.7ex>+}

\usepackage{geometry}
\usepackage[applemac]{inputenc} 
\usepackage{amsfonts}
\usepackage{amssymb}
\usepackage{epstopdf}

\newtheorem{theorem}{Theorem}[section]

\newtheorem{proposition}{Proposition}[section]
\newtheorem{corollary}{Corollary}[section]

\newtheorem{definition}{Definition} [section]

\newtheorem{lemma}{Lemma}[section]
\newtheorem{remark}{Remark}[section]

\begin{document}

\title[The $\mathfrak{S}(\infty)$ Yang-Mills measure]{A combinatorial theory of random matrices III: \\ Random walks on $\mathfrak{S}(N)$, ramified coverings and the $\mathfrak{S}(\infty)$ Yang-Mills measure}
\author[Franck Gabriel]{Franck Gabriel \\ E-mail: \parbox[t]{0.45\linewidth}{\texttt{franck.gabriel@normalesup.org}}}
\email{franck.gabriel@normalesup.org}
\address{ }
\address{Université Pierre et Marie Curie (Paris 6)\\ Laboratoire de Probabilités et Modèles Aléatoires\\ 4, Place Jussieu\\  F-75252 Paris Cedex 05.}
\address{Mathematics Institute\\ 
Zeeman Building\\ 
University of Warwick\\ 
Coventry CV4 7AL.}
\maketitle

\title{Convergence of ramified coverings, \\the $\mathfrak{S}(\infty)$-Yang-Mills.}
\maketitle

\begin{abstract}
The aim of this article is to study some asymptotics of a natural model of random ramified coverings on the disk of degree $N$. We prove that the monodromy field, called also the holonomy field, converges in probability to a non-random continuous field as $N$ goes to infinity. In order to do so, we use the fact that the monodromy field of random uniform labelled simple ramified coverings on the disk of degree $N$ has the same law as the $\mathfrak{S}(N)$-Yang-Mills measure associated with the random walk by transpositions on $\mathfrak{S}(N)$. 

This allows us to restrict our study to random walks on $\mathfrak{S}(N)$: we prove theorems about asymptotics of random walks on $\mathfrak{S}(N)$  in a new framework based on the geometric study of partitions and the Schur-Weyl-Jones's dualities. In particular, given a sequence of conjugacy classes $(\lambda_N \subset \mathfrak{S}(N))_{N \in \mathbb{N}}$, we define a notion of convergence for $(\lambda_N)_{N \in \mathbb{N}}$ which implies the convergence in non-commutative distribution and in $\mathcal{P}$-distribution of the $\lambda_N$-random walk to a multiplicative $\mathcal{P}$-L\'{e}vy process. This limiting process is shown not to be a free multiplicative L\'{e}vy process and we compute its log-cumulant transform. We give also a criterion on $(\lambda_N)_{N \in \mathbb{N}}$ in order to know if the limit is random or not. 
\end{abstract}

\section{Intoduction}

Yang-Mills theory was introduced by Yang and Mills, in 1954, in \cite{YM54} as a theory of random connections on a principal bundle with gauge symmetry. In two dimensions, it has been defined by mathematicians (\cite{Albquestion}, \cite{AHKH86}, \cite{AHKH88}, \cite{Dri89}, \cite{Dri91}, \cite{Gro85}, \cite{Gro88}, \cite{Levythese}, \cite{Levy1},  \cite{Sen92}, \cite{Sen97a}) and it has become well understood that it was a theory of random multiplicative functions from the set of paths of a two dimentional surface to a compact group $G$. In \cite{Gabholo}, the author proved that an axiomatic formulation of planar Yang-Mills measures, similar to the axioms for L\'{e}vy processes, could be set: this allowed to prove a correspondence between Yang-Mills measures and a set of L\'{e}vy processes on $G$. In the following, by Yang-Mills measure, we consider the one given by choosing a Brownian motion on $G$. 

When the structure group $G$ is a discrete group, T. L\'{e}vy proved in \cite{Levy1} that the Yang-Mills measure could be seen as the random monodromy field of a random ramified $G$-bundle. Since ramified $\mathfrak{S}(N)$-bundles are in bijection with ramified coverings with $N$ sheets, one recovers the link explained by A. D'Adda and P. Provero in \cite{addapro2} and \cite{addapro1} between $\mathfrak{S}(N)$-Yang-Mills measure and random branched $\mathfrak{S}(N)$ coverings. It has to be noticed that this link is different from the $U(N)$-Yang-Mills measure/ramified coverings partly explained in \cite{Levy2} and also known as the Yang-Mills/String duality. The theory of random ramified coverings has also some interesting and challenging links with quantum gravity \cite{Zvonkine}.

In this article, we study the asymptotic of the theory of random ramified coverings coming from the $\mathfrak{S}(N)$-Yang-Mills measure as $N$ goes to infinity: we construct the $\mathfrak{S}(\infty)$-master field. The rigorous study of the asymptotics of Yang-Mills measures driven by the Brownian motion on the unitary group begun with M. Anshelevich and A.N. Sengupta in \cite{SenguptaA} where the convergence was proved for a weak Yang-Mills measure and T. L\'{e}vy in \cite{Levymaster} where asymptotics and Makeenko-Migdal equations were proved for the full Yang-Mills measure. In this last article, the unitary, orthogonal and sympleptic groups were considered. The convergence of the Yang-Mills measure driven by the different Brownian motions, as the dimension of the group goes to infinity, was proved by using estimates for the speed of convergence in non-commutative distribution of arbitrary words in independent Brownian motions. In the article \cite{FGA}, the author and his co-authors show how to prove asymptotics of Yang-Mills measures driven by L\'{e}vy processes on the unitary and orthogonal groups without using any estimates for the speed of convergence: the asymptotic of Yang-Mills measures is a consequence of the convergence in non-commutative distribution of the L\'{e}vy processes considered and a kind of two-dimensional Kolmogorov's continuity theorem proved by T. L\'{e}vy in \cite{Levy1}.

Using similar arguments, we prove convergence of Yang-Mills measures driven by random walks on the symmetric groups by proving the convergence in non-commutative distribution of some continuous-time random walks on the symmetric groups. For sake of simplicity, for any integer $N$, we only consider random walks which jump by an element which is drawn uniformly from a conjugacy class $\lambda_{N}$ of $\mathfrak{S}(N)$. If the conjugacy class $\lambda_{N}$ converges in some sense, then the random walks will converge in non-commutative distribution. In particular, the eigenvalue distribution will converge as $N$ goes to infinity. When $\lambda_N$ is the set of transpositions, this result was shown using representation theory \cite{BeresKozma}. Besides, it seems possible that some of these results could be deduced from the proofs of articles like \cite{BeresDurrett}, \cite{Beres} where the distance from the identity was proved to converge. The study of some asymptotics linked with random walks was also one of the concern of the article \cite{Schramm}. In these last articles, the heuristic idea was to consider the symmetric group as a ``Lie group" whose ``Lie algebra" would be in some sense $\mathbb{Z}[C]$ where $C=\cup_{k=2}^{\infty} \left([|1,N|]^{k}/\sim\right)$ where $(i_1,...,i_k) \sim (j_1,...,j_k)$ if the second one is obtained by a cyclic permutation of the first one. In this picture, the exponential of $c \in C$ would just be the permutation which has $c$ as a single non-trivial cycle. The interesting fact about this is that, one can link easily the Brownian motion on the ``Lie algebra" which drives the Brownian motion on the symmetric group (the random walk by transposition) with some Erd\"{o}s-Renyi random graph process. Using the natural coupling between the two processes, one can then transfer results from Erd\"{o}s-Renyi random graph processes to the study of random walks on the symmetric graph.

In this article, we use a generalization of the non-commutative probability ideas, constructed in \cite{Gab1} and \cite{Gab2}, in order to prove asymptotics and phase transitions for the random walks on the symmetric group without using the coupling with the Erd\"{o}s-Renyi random graph processes. This allows us to show that asymptotics of random walks on the symmetric groups can be studied with the same tools than the one used for the study of multiplicative unitary L\'{e}vy processes (Section 7 of \cite{Gab2}). This method is a generalization of the method used by T. L\'{e}vy in \cite{Levy2} or \cite{Levymaster} in order to study the large $N$ asymptotics of the Brownian motions on $U(N), O(N)$ and $Sp(N)$. In particular, we do not use any theory of representations, as opposed to \cite{BeresKozma} where some results were given for the random walk by transpositions. We prove results in a more general setting, in particular we do not ask that the elements of $\lambda_N$ have bounded support as $N$ goes to infinity. This allows us to show that there exist two behaviors for the eigenvalue distribution: if the size of the support is $o(N)$ then it converges in probability to a non random probability measure we are able to compute explicitely, and if the support is growing like $\alpha N$, the eigenvalue distribution converges in law to a random measure. As an application of asymptotic $\mathcal{P}$-freeness of independent matrices which are invariant by the symmetric group, we get that the whole random walk converges in distribution in non-commutative distribution to a process whose increments are not free but $\mathcal{P}$-free. This gives the first non-trivial example of multiplicative $\mathcal{P}$-L\'{e}vy process whose log-cumulant transform is computed.

\subsection{Layout}
The results we present in this article are based on the study of the asymptotic of random walks on the symmetric group (Section \ref{conv}). In Section \ref{sec:remind}, we give a summary of some definitions and results proved in \cite{Gab1} and \cite{Gab2}. The theorems about convergence of random walks on the symmetric group $\mathfrak{S}(N)$ are presented in Section \ref{generaltheorem}. After some preliminary results explained in Section \ref{prep}, we give the proofs of these theorems in Section \ref{proof}. The log-cumulant transform for the limit of random walks on the symmetric group is computed in Section~\ref{sec:logcu}. 

A short presentation of Yang-Mills measure with $\mathfrak{S}(N)$-gauge group is explained in Section \ref{YMConv}. In the same section, we prove that the Wilson loops in $\mathfrak{S}(N)$-Yang-Mills measure converge in probability as $N$ goes to infinity to a non-random field: the $\mathfrak{S}(\infty)$-master field. 

Based on the results of T. L\'{e}vy in \cite{Levy1}, we explain in Section \ref{sec:cov} how to link the study of random coverings of the disk and the study of $\mathfrak{S}(N)$-Yang-Mills measure. This allows us to prove that the monodromy field of a model of random simple ramified labelled coverings of the unit disk with $N$ sheets converges in probability to the $\mathfrak{S}(\infty)$-master field. In Section \ref{sec:computation}, we explain how to compute the $\mathfrak{S}(\infty)$-master field by giving an example. 
 
\section{Convergence of random walks on $\mathfrak{S}(N)$}
\label{conv}
\subsection{A brief reminder}
\label{sec:remind}
We review briefly the definitions and results, proved in \cite{Gab1} and \cite{Gab2}, that we will need later. In these articles, we study the asymptotics of random matrices when the size of these matrices goes to infinity. Let ${L}^{\infty^{-}} \otimes \mathcal{M}(\mathbb{C})$ be the set of sequences of random matrices $(M_N)_{N \in \mathbb{N}}$ such that $M_{N}$ is a $N \times N$ random matrice whose entries have order of any moments. For any family of elements $(a_i)_{i\in I}$ in ${L}^{\infty^{-}} \otimes \mathcal{M}(\mathbb{C})$, the algebra generated by  $(a_i)_{i\in I}$ is simply:
\begin{align}
\label{eq:algebre}
\mathcal{A}\left((a_i)_{i\in I}\right) = \{ P(a_{i_1},...,a_{i_k}) | (i_1,...,i_k) \in I, P \in \mathbb{C}\{X_1,...,X_k\}, k\in \mathbb{N}\}, 
\end{align}
where $\mathbb{C}\{X_1,...,X_k\}$ is the algebra of non-commutative polynomials.

 Let $M$ be in ${L}^{\infty^{-}} \otimes \mathcal{M}(\mathbb{C})$. We will suppose from now on that any element $M$ of ${L}^{\infty^{-}} \otimes \mathcal{M}(\mathbb{C})$ we consider is {\em invariant in law by conjugation by the symmetric group}. This means that for any integer $N$, for any $\sigma \in \mathfrak{S}(N)$, viewed as a $N\times N$ matrice, $\sigma M_N \sigma^{-1}$ has the same law as $M_N$. In order to study the asymptotic of $M$, let us define the partition observables. 

Let $k$ be a positive integer, $\mathcal{P}_k$ is the set of partitions of $\{1,...,k, 1',...,k'\}$. There exists a graphical notation for partitions as illustrated in Figure \ref{fig:1}. Let $p \in \mathcal{P}_k$, let us consider $k$ vertices in a top row, labelled from $1$ to $k$ from left to right and $k$ vertices in a bottom row, labelled from $1'$ to $k'$ from left to right. Any edge between two vertices means that the labels of the two vertices are in the same block of the partition $p$. Using this graphical point of view, the set of permutations of $k$ elements, namely $\mathfrak{S}_k$, is a subset of $\mathcal{P}_k$: if $ \sigma$ is a permutation, we associate the partition $\{\{i,\sigma(i)'\} | i \in \{1,...,k\}\}$. 

\begin{figure}
\centering
\resizebox{0.75\textwidth}{!}{
  \includegraphics[width=6pt]{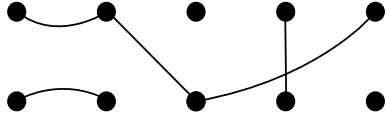}
}
\caption{The partition $\left\{ \{1',2' \}, \{1,2,3',5 \}, \{3 \}, \{4',4\},\{5' \} \right\}$.}
\label{fig:1}       
\end{figure}

 Let $p$ and $p'$ be two partitions in $\mathcal{P}_k$. The set $\mathcal{P}_k$ has some nice structures that we are going to explain: 
\begin{enumerate}
\item{\bf Transposition: } The partition $^{t}p$ is obtained by flipping along an horizontal line a diagram representing $p$. 
\item {\bf Order} $p'$ is coarser that $p$, denoted $p \trianglelefteq p'$, if any block of $p$ is included in a block of $p'$. 
\item {\bf Supremum: } $p\vee p'$ is obtained by putting a diagram representing $p'$ over one representing $p$. 
\item {\bf Multiplication: } $p \circ p'$ is obtained by putting a diagram representing $p'$ above one representing $p$, identifying the lower vertices of $p'$ with the upper vertices of $p$, erasing the vertices in the middle row, keeping the edges obtained by concatenation of edges passing through the deleted vertices. It has the nice following property: if $p \circ p' \in \mathfrak{S}_k$ then $p$ and $p'$ are in $\mathfrak{S}_k$. Doing so, we remove a certain number of connected components, number which is denoted by $\kappa(p,p')$.
\item {\bf A family of multiplications:} $\mathbb{C}[\mathcal{P}_k(N)]$ is the algebra such that $pp'=N^{\kappa(p,p')} p \circ p'.$
\item {\bf Neutral element: } The partition $id_k = \{\{i,i'\}| i \in \{1,...,k\}\}$ is a neutral element for the multiplication. Often we will denote it by $id$ when there can not be any confusion. 
\item {\bf Height function: }  ${\sf nc}(p)$ is the number of blocks of $p$. 
\item {\bf Cycle: } A cycle is a block of $p\vee id$: ${\sf nc}(p\vee id)$ is the number cycles of $p$. A partition $p$ is {\em irreducible} if ${\sf nc}(p \vee id) = 1$. 
\item {\bf Representation: } There exists an important representation of $\mathbb{C}[\mathcal{P}_k(N)]$ on $\left(\mathbb{C}^{N}\right)^{\otimes k}$. Let $(e_i)_{i=1}^{N}$ be the canonical basis of $\mathbb{C}^{N}$ and let $E_{i}^{j}$ be the matrix which sends $e_j$ on $e_i$ and any other element of the canonical basis on $0$. For any $I=(i_1,...,i_k,i_{1'},...,i_{k'}) \in \{1,...,N\}^{2k}$, we define ${\sf Ker}(I)$ the partition such that two elements $a$ and $b$ of $\{1,...,k,1',...,k'\}$ is in a block of ${\sf Ker}(I)$ if and only if $i_a = i_b$. We define: 
\begin{align*}
\rho_N(p) = \sum_{I = (i_1,...,i_k,i_{1'},...,i_{k'}) \in \{1,...,N\}^{2k}| p \trianglelefteq {\sf Ker}(I) } E_{i_{1'}}^{i_1} \otimes ... \otimes E_{i_{k'}}^{i_k}. 
\end{align*}
The application $\rho_N$ is a representation of $\mathbb{C}[\mathcal{P}_k(N)]$. We will also define: 
\begin{align*}
\rho_N(p^{c}) = \sum_{I = (i_1,...,i_k,i_{1'},...,i_{k'}) \in \{1,...,N\}^{2k}| p = {\sf Ker}(I) } E_{i_{1'}}^{i_1} \otimes ... \otimes E_{i_{k'}}^{i_k}. 
\end{align*}
\item {\bf Tensor product: } The partition $p \otimes p' \in \mathcal{P}_{2k}$ is obtained by putting a diagram representing $p'$ on the right of a diagram representing $p$. It satisfies the identity $\rho_{N}(p \otimes p') = \rho_N(p) \otimes \rho_N(p')$. 
\item {\bf Extraction: } The extraction of $p$ to a symmetric subset $I$ of $\{1,...,k,1',...,k'\}$ is obtained by erasing the vertices which are not in $I$ and relabelling the remaining vertices (definition given at the beginning of Section 3.1.2 of \cite{Gab1}).
\end{enumerate}

In \cite{Gab1}, we defined a distance (Definition $2.2$ of \cite{Gab1}) and a geodesic order on $\mathcal{P}_k$ (Definition $2.4$ of \cite{Gab1}). For any partitions $p$ and $p'$ in $\mathcal{P}_k$, the distance between $p$ and $p'$ is: 
\begin{align*}
d(p,p') = \frac{{\sf nc}(p) + {\sf nc}(p')}{2} - {\sf nc}(p\vee p'). 
\end{align*}
We wrote that $p \leq p'$ if $d(id,p)+d(p,p') = d(id, p')$: this defines an order on $\mathcal{P}_k$. The set $\{p| p\leq p'\}$ is denoted by $[id, p']_{\mathcal{P}_k}$.

Using the set $\mathcal{P}_k$, we defined some observables. Let $N$ be a positive integer and let $p \in \mathcal{P}_k$, the {\em random $p$-moment} of $M_N$ is given by: 
\begin{align}
\label{eq:mom}
m_{p}(M_N) = \frac{1}{N^{{\sf nc}(p \vee id)}} Tr \left( M_N^{\otimes k} \rho_N(^{t}p)\right),
\end{align}
and the {\em $p$-moment} of $M_N$ is given by $\mathbb{E}m_p(M_N) = \mathbb{E}\left[m_p(M_N)\right].$ The {\em random $p$-exclusive moment} and the {\em $p$-exclusive moment} of $M_N$ are given by the same formulas, except that one has to replace $p$ with $p^{c}$. 

Using the Schur-Weyl-Jones duality, $\mathbb{E}\left[M_N^{\otimes k}\right]$ belongs to ${\sf Vect} \left\{\rho_N(p) | p \in \mathcal{P}_k\right\}$ and if $N$ is bigger than $2k$, $\left\{\rho_N(p) | p \in \mathcal{P}_k\right\}$ is a basis. We will always suppose from now on that this condition on $N$ and $k$ is always satisfied. In \cite{Gab2}, we defined the {\em finite dimensional $\mathcal{P}$-cumulants} $\left(\mathbb{E}\kappa_{p}(M_N)\right)_{p \in \mathcal{P}_k}$ by the formula: 
\begin{align}
\label{eq:coord}
\mathbb{E}\left[M_N^{\otimes k}\right] = \sum_{p \in \mathcal{P}_k} \frac{\mathbb{E}\kappa_{p}(M_N)}{N^{{\sf nc}(p)- {\sf nc}(p\vee id)}}\rho_N(p). 
\end{align}

For any $M_1,...,M_k$ in  ${L}^{\infty^{-}} \otimes \mathcal{M}(\mathbb{C})$, one can generalize these definitions in order to define for any $p$ the observables $m_{p}(M_{1,N}, ..., M_{k,N})$, $\mathbb{E}m_{p}(M_{1,N}, ..., M_{k,N})$, $m_{p^{c}}(M_{1,N}, ..., M_{k,N})$, $\mathbb{E}m_{p^{c}}(M_{1,N}, ..., M_{k,N})$ and  $\mathbb{E}\kappa_{p}(M_{1,N}, ..., M_{k,N})$: one has to replace $M_N^{\otimes k}$ by $M_{1,N}\otimes...\otimes M_{k,N}$. Let us state a consequence of Theorems $3.2$ and $4.2$ of \cite{Gab2}. 

\begin{theorem}
\label{th:convergence}
The three following convergences as $N$ goes to infinity are equivalent: 
\begin{enumerate}
\item convergence of $\left(\mathbb{E}\kappa_{p}(M_N)\right)_{p \in \mathcal{P}_k}$,
\item convergence of $\left(\mathbb{E}m_p(M_N)\right)_{p \in \mathcal{P}_k}$, 
\item convergence of $\left(\mathbb{E}m_{p^{c}}(M_N)\right)_{p \in \mathcal{P}_k}$. 
\end{enumerate}
If for any integer $k$, one of these convergences holds, we say that $M$ converges in $\mathcal{P}$-distribution. Let us suppose it is the case: for any integer $k$, for any $p \in \mathcal{P}_k$: 
\begin{align*}
\lim_{N \to \infty} \mathbb{E}m_{p}(M_N) &= \sum_{p'\in[id, p]_{\mathcal{P}_k}} \lim_{N \to \mathbb{N}} \mathbb{E}\kappa_{p'}(M_N), \\
\lim_{N \to \infty} \mathbb{E}m_{p}(M_N) &= \sum_{p' \in \mathcal{P}_k| p \trianglelefteq p', {\sf nc}(p\vee id) = {\sf nc}(p'\vee id)} \lim_{N \to \infty} \mathbb{E}m_{p'^{c}} (M_N),
\end{align*}
and for any $\sigma \in \mathfrak{S}_k$: 
\begin{align*}
\lim_{N \to \infty} \mathbb{E}m_{\sigma^{c}}(M_N) = \lim_{N \to \infty} \mathbb{E}\kappa_{\sigma}(M_N).
\end{align*} 
\end{theorem}

The same results hold for a family of elements of $L^{\infty^{-}} \otimes \mathcal{M}(\mathbb{C})$. From now on, when in an observable or in a coordinate number we do not put the index $N$, it means that we take the limit as $N$ goes to infinity. For example, $\mathbb{E}m_{p}(M)$ stands for $\lim_{N \to \infty} \mathbb{E}m_{p}(M_N)$. Let us suppose that $M$ converges in $\mathcal{P}$-distribution. We associate the {\em $\mathcal{R}$-transform} of $M$ which is the element $\mathcal{R}[M] \in \left(\bigoplus_{k=0}^{\infty} \mathbb{C}[\mathcal{P}_k]\right)^{*}$ such that for any $k$, any $p \in \mathcal{P}_k$, $\left(\mathcal{R}[M]\right)(p) = \mathbb{E}\kappa_{p}(M)$. 

We say that $M$ satisfies the {\em asymptotic $\mathcal{P}$-factorization} if for any positive integers $k$ and $k'$, any $p_1 \in \mathcal{P}_k$ and $p_2 \in \mathcal{P}_{k'}$, 
\begin{align*}
\mathbb{E}m_{p_1 \otimes p_2} (M) = \mathbb{E}m_{p_1} (M)\mathbb{E}m_{p_2} (M). 
\end{align*}
Let us state an easy version of Theorem $3.1$ of \cite{Gab2}.

\begin{theorem}
\label{th:conv}
Let us suppose that $(M_i)_{i \in I} \in L^{\infty^{-}} \otimes \mathcal{M}(\mathbb{R})$ converges in $\mathcal{P}$-distribution and satisfies the asymptotic $\mathcal{P}$-factorization property then the $\mathcal{P}$-moments of $(M_i)_{i \in I}$ converge in probability: for any integer $k$, any $i_1,...,i_k$ in $I$, any $p \in \mathcal{P}_k$, $m_{p}(M_{i_1,N}, ..., M_{i_k,N})$ converges in probability to $\mathbb{E}m_{p}(M_{i_1},...,M_{i_k})$. 
\end{theorem}

In particular, if $M$ is an orthogonal element, which means that for any positive integer $N$, $M_{N}$ is orthogonal, then the eigenvalue distribution of $M_N$ is on the unit circle $\mathbb{U}$. If $M$ converges in $\mathcal{P}$-distribution, the mean eigenvalue distribution converges to a measure $\mu$ (Theorem $1.1$ of \cite{Gab2}) and if $M$ satisfies the asymptotic $\mathcal{P}$-factorization property, the eigenvalue distribution converges in probability to~$\mu$.

Let $M$ and $L$ be elements of ${L}^{\infty^{-}} \otimes \mathcal{M}(\mathbb{C})$ which converge in $\mathcal{P}$-distribution. The notion of $\mathcal{P}$-freeness is defined as a condition of vanishing of mixed cumulants and a factorization property for compatible cumulants. The elements $M$ and $L$  are $\mathcal{P}$-free if the two following conditions hold:
\begin{itemize}
\item for any integer $k$, any $p \in \mathcal{P}_k$, any $(B_{1},...,B_{k}) \in \{M,L\}^{k}$,  $\mathbb{E}\kappa_{p}(B_{1},...,B_k) = 0$ if there exists $i$ and $j$ in $\{1,...,k\}$ in the same cycle of $p$ such that $B_{i} \ne B_{j}$
\item for any integers $k$ and $k'$, any $p \in \mathcal{P}_k$, any $p' \in \mathcal{P}_{k'}$, $\mathbb{E}\kappa_{p\otimes p'}(M,...,M,L,...,L)$ is equal to $\mathbb{E}\kappa_{p}(M,...,M)\mathbb{E}\kappa_{p'}(L,...,L).$
\end{itemize}  In \cite{Gab2}, we proved the following theorem. We recall that we always suppose that the random matrices are invariant in law by conjugation by the symmetric group. 

\begin{theorem}
\label{th:free}
If for any positive integer $N$, $M_N$ and $L_N$ are independent then they are $\mathcal{P}$-free. 
\end{theorem}

The notion of $\mathcal{P}$-freeness differs from the notion of freeness in the sense of Voiculescu. In particular, we gave a condition, in Theorem $3.7$ of \cite{Gab2}, in order to prove that two $\mathcal{P}$-free elements are not free. Let us denote by ${\sf 0}_{2}$ the partition $\{\{1,2,1',2'\} \}$.

\begin{lemma}
\label{eq:nonlibre}
Let us suppose that $M$ and $L$ are $\mathcal{P}$-free, satisfy the asymptotic $\mathcal{P}$-factorization and that $$\mathbb{E}\kappa_{{\sf 0}_2}[M] \mathbb{E}\kappa_{{\sf 0}_2}[L] \ne 0,$$ then $M$ and $L$ are not free. 
\end{lemma}

In \cite{Gab2}, we proved a general theorem about the convergence of matricial L\'{e}vy processes in $\mathcal{P}$-distribution. We will explain a consequence which is needed in this article. Let us consider for any $N$, $(X_t^{N})_{t \geq 0}$ a multiplicative real-matricial L\'{e}vy process which is invariant in law by conjugation by $\mathfrak{S}(N)$. Let us denote, for any positive integers $k$ and~$N$: 
\begin{align*}
G_k^{N} = \frac{d}{dt}_{|t=0}\mathbb{E}\left[\left(X_t^{N}\right)^{\otimes k}\right]. 
\end{align*}
The endormorphism $G_k^{N}$ is also in ${\sf Vect} \left\{\rho_N(p) | p \in \mathcal{P}_k\right\}$ and we can define the coordinate numbers of $G_k^{N}$, $\kappa_{p}(G_k^{N})$ by the same formula as Equation (\ref{eq:coord}). We can also define the moments of $G_k^{N}$, $m_{p}(G^{N}_k)$, by the same formula as Equation (\ref{eq:mom}), and the exclusive moments of $G_k^{N}$, $m_{p^{c}}(G^{N}_k)$ by replacing $p$ by $p^{c}$. We say that the sequence $(G^{N}_k)_{N \in \mathbb{N}}$ converges if and only if its coordinate numbers converge. This is equivalent to say that the $\mathcal{P}$-moments or the $\mathcal{P}$-exclusive moments of $(G^{N}_k)_{N \in \mathbb{N}}$ converge. If for any positive integers $k$ and $l$, for any partitions $p$ and $p'$ respectively in $\mathcal{P}_k$ and $\mathcal{P}_l$, 
\begin{align*}
\lim_{N \to \infty}m_{p\otimes p'}(G_{k+l}^{N}) = \lim_{N \to \infty} m_{p}(G_k^{N}) +  \lim_{N \to \infty} m_{p'}(G_l^{N}),
\end{align*}
we say that $(G_k^{N})_{k,N}$ {\em condensates weakly}. Recall the notion of multiplicative $\mathcal{P}$-L\'{e}vy process defined in Definition 2.9 of \cite{Gab2}. Let us state a direct consequence of Theorem $6.1$ and Remark 6.2 of \cite{Gab2}. 

\begin{theorem}
\label{th:convlevy}
Let us suppose that for any positive integer $k$, $(G_k^{N})_{N > 0}$ converges. The process $(X_t^{N})_{t \geq 0}$ converges in $\mathcal{P}$-distribution toward the $\mathcal{P}$-distribution of a $\mathcal{P}$-L\'{e}vy process. For any $p \in \mathcal{P}_k$, for any $t_0 \geq 0$, 
\begin{align*}
\frac{d}{dt}_{| t = t_0}\mathbb{E}\kappa_p(X_t) = \sum_{p_1 \in \mathcal{P}_k, p_2 \in \mathcal{P}_k | p_1\circ p_2 = p, p_1\prec p } \kappa_{p}(G_k) \mathbb{E}\kappa_{p_2}(X_{t_0}). 
\end{align*}
Besides, the process $(X_t^{N})_{t \geq 0}$ satisfies the asymptotic $\mathcal{P}$-factorization if and only if $(G^{N}_{k})_{k,N}$ weakly condensates. If so, it converges in probability in $\mathcal{P}$-distribution.
\end{theorem}

We will not explain the definition of the notation $p_1 \prec p$: the only important fact to know for this article is Lemma $3.6$ of \cite{Gab1} which states that for any $\sigma \in \mathfrak{S}_k$: 
\begin{align*}
\{ p' \in \mathcal{P}_k | p' \prec \sigma \} =[id, \sigma]_{\mathcal{P}_k} \cap \mathfrak{S}_k.  
\end{align*}

\subsection{General theorems of convergence}
\label{generaltheorem}

In this section, we state the general theorems about convergence of random walks on the symmetric group that we are going to prove in this article. The proofs will be given in Section \ref{proof}. 

Let $N$ be a positive integer, let us consider $\lambda_N$ a conjugacy class of $\mathfrak{S}(N)$, the symmetric group on $N$ elements. We denote by $\#\lambda_N$ the size of the conjugacy class $\lambda_N$. Let $\sigma$ be in $\lambda_N$ and let $i $ be in $\{1,...,N\}$. For any $k \in \{1,...,N\}$, the {\em period} of $k$ in $\sigma$ is the smallest positive integer $n$ such that $\sigma^{n}(k)=k$. We denote by $\lambda_{N}(i)$ the number of elements in $\{1,...,N\}$ which period in $\sigma$ is equal to $i$: this number does not depend on the choice of $\sigma$. Thus, we can see $\lambda_N$ as a way to decompose the integer $N$: $\lambda_N$ can be coded by the sequence $\left(\lambda_N(i)\right)_{i = 1}^{\infty}$ which satisfies $\sum_{i=1}^{\infty} \lambda_N(i)= N.$

\begin{definition}
\label{simplewalk}
We define the $\lambda_N$-random walk on $\mathfrak{S}(N)$, denoted by $\left(S_{t}^{N}\right)_{t \geq 0}$, as the Markov process on $\mathfrak{S}(N)$ such that $S_0^{N} = id_N$ and whose generator is given by:  
\begin{align*}
\forall f \in \mathbb{R}^{\mathfrak{S}(N)}, \forall \sigma_0 \in \mathfrak{S}(N), H_{N} f (\sigma_0)= \frac{N}{\lambda_N(1^{c})}\frac{1}{\# \lambda_N} \sum_{\sigma \in \lambda_N} \big[ f(\sigma \sigma_0) - f(\sigma_0) \big], 
\end{align*}
where we used the following notation: 
\begin{align*}
\lambda_N(1^{c}) = N-\lambda_N(1). 
\end{align*}
\end{definition}

This random walk is invariant in law by conjugation by $\mathfrak{S}(N)$. 

\begin{lemma}
\label{invariance}
Let $\sigma$ be in $\mathfrak{S}(N)$, let $\left(S_{t}^{N}\right)_{t \geq 0}$ be a $\lambda_N$-random walk on $\mathfrak{S}(N)$. We have the equality in law: $\left(\sigma S_{t}^{N} \sigma^{-1} \right)_{t \geq 0}= \left( S_{t}^{N}\right)_{t \geq 0}.$
\end{lemma}

\begin{proof}
It is a consequence of the fact that the generator of $\left(S_{t}^N\right)_{t \geq 0}$ is invariant by conjugation by $\mathfrak{S}(N)$.
\qed \end{proof}

Let us consider a sequence $\left(\lambda_N\right)_{N \in \mathbb{N}^{*}}$ such that, for any positive integer $N$, $\lambda_N$ is a conjugacy class of $\mathfrak{S}(N)$.
\begin{definition}
\label{convergencedef}
 The sequence $(\lambda_N)_{N \in \mathbb{N}}$ converges if and only if there exists: 
 \begin{align*}
 \left( \lambda(i)\right)_{i \geq 2} \in \left\{ \left(a_i\right)_{i \in \mathbb{N} \setminus \{0,1\}} | \forall i \geq 2, a_i \geq 0, \sum_{i=2}^{\infty} a_i \leq 1\right\}
 \end{align*}
  such that for any integer $i \geq 2$: 
\begin{align*}
\frac{\lambda_N(i)}{\lambda_{N}(1^c)} \underset{N \to \infty}{\longrightarrow} \lambda(i),
\end{align*}
and there exists $\alpha \in [0,1]$ such that: 
\begin{align*}
\frac{\lambda_N(1)}{N}  \underset{N \to \infty}{\longrightarrow} 1-\alpha. 
\end{align*}
We will denote by $\lambda(1)$ the value of $1-\alpha$. The sequence $(\lambda_N)_{N \in \mathbb{N}}$ is {\em evanescent} if $\alpha = 0$ and it is {\em macroscopic} if $\alpha >0$. 
\end{definition}

For any positive integer $N$, let us consider $\left(S_t^{N}\right)_{t \geq 0}$ a $\lambda_{N}$-random walk on $\mathfrak{S}(N)$. Recall the notion of exclusive moments $\mathbb{E}m_{\sigma^{c}}$. From now on, let us suppose that the sequence $(\lambda_N)_{N \in \mathbb{N}}$ converges. 

\begin{theorem}\label{convergence1}
For any $t \geq 0$, the {\em mean} eigenvalue distributions of $\left(S_{t}^{N}\right)_{N \in \mathbb{N}}$ converge as $N$ goes to infinity to a probability measure $\overline{\mu}_{t}^{\lambda}$ which has the form: 
\begin{align*}
\overline{\mu}_{t}^{\lambda} = \sum_{n \in \mathbb{N}^{*}} \sum_{k=0}^{n-1} \frac{m_{n^c}(t)}{n}\delta_{e^{\frac{2ik\pi}{n}}} + m_{\infty^{c}}(t) \lambda_{\mathbb{U}}, 
\end{align*}
with $m_{\infty^{c}}(t) = 1-\sum_{k=0}^{\infty} m_{n^{c}}(t) \geq 0$,  $m_{n^{c}}(t) \geq 0$ for any integer $n$ and $\lambda_{\mathbb{U}}$ is the uniform probability measure on the unit circle $\mathbb{U}$.

We denote by  $[\sigma]$ the conjugacy class of $\sigma$ which, as already noticed, can be seen as a decomposition of $N$. Also, by convention, for any positive integer $k$: $\frac{\left((1-0)^{k} - 1\right)}{0} = -k$. Let us consider the unique solution $\left(\left(m_{\sigma^{c}}(t)\right)_{\sigma \in \cup_k \mathfrak{S}_k}\right)_{t \geq 0}$ of the system of differential equations: 

$\forall k \in \mathbb{N}^{*}$, $\forall \sigma_0 \in \mathfrak{S}_k$, $\forall t_0 \geq 0$: 
\begin{align*}
\frac{d}{dt}_{| t = t_0}\!\!\!\! m_{\sigma_0^{c}}(t)\!= \!\frac{\left((1-\alpha)^{k} - 1\right)}{\alpha}&m_{\sigma_0^{c}}(t_0) \\ &\!\!\!\!\!\!\!\!\!\!\!\!\!\!\!\!\!\!\!\!\!\!\!\!\!\!\!\!\!\!\!\!\!\!\!\!\!\!\!\!\!\!\!\!\!\!\!\!\!\!+\!\!\!\sum_{ \sigma \in \mathfrak{S}_k \setminus \{id\},  \sigma \leq \sigma_0} \!\!\!\!\!\!\!\!\!\!\!\!\! \alpha^{{\sf nc}(\sigma \vee id) - [\sigma](1) - 1}\! \left(\prod_{i=2}^{k} \left(\lambda(i)\right)^{\frac{[\sigma](i)}{i}} \right)\!\! (1-\alpha)^{[\sigma](1)} m_{(^{t}\sigma \sigma_0)^{c}}(t_0), 
\end{align*}
with the initial conditions: $ \forall k \in \mathbb{N}^{*}, \forall \sigma \in \mathfrak{S}_k, 
m_{\sigma^{c}}(0) = \delta_{\sigma = id}.$
Then for any positive integer $n$, for any real $t \geq 0$: 
\begin{align*}
m_{n^{c}}(t) = m_{(1,...,n)^{c}}(t), 
\end{align*}
where $(1,...,n) \in \mathfrak{S}_n$ is a $n$-cycle. Besides for any integer $k$, any $t \geq 0$ and any $\sigma \in \mathfrak{S}_k$: 
\begin{align*}
m_{\sigma^{c}}(t) = \lim_{N \to \infty}\mathbb{E}m_{\sigma^{c}}\left[S_t^{N}\right].
\end{align*}
\end{theorem}

In fact, in the proof of Theorem \ref{convergence1}, we will show that the eigenvalue distributions of $\left( S_{t}^{N}\right)_{ N \in \mathbb{N}}$ converge in law as $N$ goes to infinity to a random measure: depending on the behavior of $(\lambda_N)_{N \in \mathbb{N}}$, one can know if the limit is or is not random. 

\begin{theorem}
\label{converfac}
Let $t$ be a positive real. The eigenvalue distribution $\mu_{t}^{\lambda_N}$ of $S_t^{N}$ converges in law to a random measure on $\mathbb{U}$, denoted by $\mu_{t}^{\lambda}$. Two behaviors are possible: 
\begin{enumerate}
\item if the sequence $(\lambda_N)_{N \in \mathbb{N}}$ is evanescent then the limiting measure is a non-random measure on $\mathbb{U}$, $\mu_{t}^{\lambda} = \overline{\mu}_{t}^{\lambda}$, and the convergence holds in probability. The family $\left(S_{t}^{N}\right)_{N \in \mathbb{N}}$ satisfies the asymptotic $\mathcal{P}$-factorization and thus it converges in probability in $\mathcal{P}$-distribution, 

\item if the sequence $(\lambda_N)_{N \in \mathbb{N}}$ is macroscopic, then the limiting measure is not a non-random measure on $\mathbb{U}$ and the family $\left(S_{t}^{N}\right)_{N \in \mathbb{N}}$ does not satisfy the asymptotic $\mathcal{P}$-factorization. 
\end{enumerate}
\end{theorem}

When the sequence $(\lambda_N)_{N \in \mathbb{N}}$ is evanescent, we can compute explicitly the measure $\mu_t^{\lambda}$. Given Theorem \ref{convergence1}, we only need to compute its moments or $m_{n^{c}}(t)$. In the following theorem, we will use the same notations as in Theorem \ref{convergence1} and in Definition \ref{convergencedef}.

\begin{theorem}
\label{permutations}
Let us suppose that $(\lambda_N)_{N \in \mathbb{N}}$ is evanescent. Let $n$ be a positive integer and let $t \geq 0$. We have: 
\begin{align}
\label{equationexclu}
m_{n^{c}}(t) = e^{-nt} \sum_{k=0}^{n-1} t^{k} \frac{n^{k-1}}{k!} \sum_{(i_1,...,i_k)\in \mathbb{N}^{*}, \sum_{j=1}^{k}i_j = n-1} \prod_{j=1}^{k} \lambda(i_j+1). 
\end{align}
We used the usual conventions for the products and the sums, thus $m_{1^{c}}(t)= e^{-t}$. 
As a consequence, for any positive integer $n$, and any $t \geq 0$: 
\begin{align*}
m_{n}(t)\!:=\! \int_{z \in \mathbb{U}} \! z^{n} d\overline{\mu}_{t}^{\lambda}(z)\! =\! \sum_{d| n} e^{-dt} \sum_{k=0}^{d-1} t^{k} \frac{d^{k-1}}{k!} \!\!\sum_{(i_1,...,i_k)\in \mathbb{N}^{*}, \sum_{j=1}^{k}i_j = d-1} \prod_{j=1}^{k} \!\lambda(i_j+1). 
\end{align*} 
\end{theorem}

In particular, let us consider a positive integer $k$, let us suppose that $\lambda(k)=1$ and for any positive integer $l \ne k$, $\lambda(l)=0$. This means that we are considering a random walk which jumps by multiplication by a uniform $k$-cycle. Let $t$ be a non-negative real. If there does not exist any positive integer $u$ such that $ n = u(k-1)+1$, then $m_{n^{c}}(t) = 0$. Besides, for any $u \in \mathbb{N}$: 
\begin{align}
\label{momexspe}
m_{( u(k-1)+1)^{c}} (t) = e^{-(u(k-1)+1)t} t^{u} \frac{(u(k-1)+1)^{u-1}}{u!}. 
\end{align}

Theorems \ref{convergence1} and \ref{converfac} can be extended to the whole process $\left(S_t^{N}\right)_{t \geq 0}$.

\begin{theorem}
\label{converlevy}
The family $\left(S_{t}^{N}\right)_{t \geq 0}$ converges in $\mathcal{P}$-distribution as $N$ goes to infinity toward the $\mathcal{P}$-distribution of a multiplicative $\mathcal{P}$-L\'{e}vy process. If $\left(\lambda_{N}\right)_{N \in \mathbb{N}}$ is evanescent, it is not a multiplicative free  L\'{e}vy process: the multiplicative increments are not asymptotically free. 
\end{theorem}

In Section  \ref{sec:logcu}, we will prove an other theorem, namely Theorem \ref{th:logcum}, which allows us to compute the log-cumulants of the limiting multiplicative $\mathcal{P}$-free L\'{e}vy process.

In the second part of the article, in order to construct the Yang-Mills $\mathfrak{S}_{\infty}$-field, we will need the following result whose proof will not be given since it is an easy consequence of Theorems \ref{converfac} and \ref{converlevy}, Lemma \ref{invariance} of this article and Theorem 5.1 and 3.1 and Proposition 3.1 of~\cite{Gab2}.

Let $n$ be a positive integer and for any positive integer $N$, let us consider $(S_t^{1,N})_{t \geq 0}$,...,$(S_t^{n,N})_{t \geq 0}$, $n$ independent $\lambda_{N}$-random walks on $\mathfrak{S}(N)$. Recall the Equation (\ref{eq:algebre}) and the representation $\rho_N$ defined in Section \ref{sec:remind}.   
\begin{theorem}
\label{convergencefamille}
The family of random matrices $${\bold{F}}_N = \bigcup\limits_{k \in \{1,...,n\}} \bigg(\big(S_t^{k,N}\big)_{t \geq 0} \cup \big((S_t^{k,N})^{-1}\big)_{t \geq 0}\bigg)$$ 
converges in $\mathcal{P}$-distribution as $N$ goes to infinity. Let us suppose that $\left(\lambda_N\right)_{N \in \mathbb{N}}$ is evanescent. The family $\bold{F}_N$ satisfies the asymptotic $\mathcal{P}$-factorization property: $\bold{F}_N$ converges in probability in $\mathcal{P}$-distribution. Besides, for any $i \in \{1,...,k\}$, $\big(S_t^{i,N}\big)_{t \geq 0} \cup \big((S_t^{i,N})^{-1}\big)_{t \geq 0}$ is asymptotically $\mathcal{P}$-free but not asymptotically free from $\bigcup\limits_{k \in \{1,...,n\}\setminus \{i\}} \bigg(\big(S_t^{k,N}\big)_{t \geq 0} \cup \big((S_t^{k,N})^{-1}\big)_{t \geq 0}\bigg).$
\end{theorem}

We will finish this section with some results about phase transition for these random walks. Let us suppose for this discussion that $(\lambda_N)_{N \in \mathbb{N}}$ is evanescent. In Theorem \ref{convergence1}, we saw that the measure $\mu_{t}^{\lambda}$ is the sum of an atomic part and $m_{\infty^{c}}(t)$ times the Lebesgue measure on the unit circle. In fact there exists a real $t^{\lambda}_c \geq 0$ such that ${\mu}_{t}^{\lambda}$ is purely atomic for $t \leq t^{\lambda}_c$ and for any $t >t_c$, ${\mu}_{t}^{\lambda}$ is a sum of a purely atomic and a multiple of the Lebesgue measure. This critical time is the same critical as found by N. Berestycki in  \cite{Beres} for the phase transition for the distance to the identity. 

\begin{theorem}
\label{theo:tcritique}
Let us suppose that $(\lambda_N)_{N \in \mathbb{N}}$ is evanescent. The function which sends $t$ on $m_{\infty^{c}}(t)$ is continuous and converges to $1$ as $t$ goes to infinity. Besides, if we define: 
\begin{align}
\label{tcrit}
t^{\lambda}_c= \delta_{\sum_{j=2}^{\infty} \lambda(j) = 1} \frac{1}{ \sum_{j=2}^{\infty} (j-1) \lambda(j)}, 
\end{align}
for any $0\leq t\leq t_c^{\lambda}$, $m_{\infty^{c}}(t) = 0$ and for any $t> t_c^{\lambda}$, $m_{\infty^{c}}(t) > 0$. 
\end{theorem}

Using the theorems already explained, we get a generalization of Theorem $3$ of \cite{BeresDurrett}, Theorem $4$ of \cite{BeresKozma}. We recommend the reader to have also a look at Theorem $3$ of \cite{Beres}.

\begin{corollary} 
\label{KozmaBerestycki}
Let us suppose that $(\lambda_N)_{N \in \mathbb{N}}$ converges and is evanescent. For any positive integer $N$, let $\left(S_t^{N}\right)_{t \geq 0}$ be a $\lambda_N$-random walk on the symmetric group. For any permutations $\sigma$ and $\sigma'$ in $\mathfrak{S}(N)$, let $d_{\mathfrak{S}(N)}\left(\sigma,\sigma'\right) $ be the distance in $\mathfrak{S}(N)$ between $\sigma$ and $\sigma'$ defined as $N - {\sf nc}\left(\sigma \vee \sigma' \right)$. Then for any $t \geq 0$, $\frac{1}{N}  d_{\mathfrak{S}(N)} \left(id_N, S_t^{N}\right)$ converges in probability as $N$ goes to infinity to the non-random continuous function:
\begin{align*}
d^{\lambda}(t) = 1-\sum_{k=1}^{\infty} \frac{1}{k}m_{k^{c}}(t),  
\end{align*}
where $m_{k^{c}}(t)$ is given by Equation $(\ref{equationexclu})$. 
\end{corollary}

Using Equation $(\ref{momexspe})$, we recover Equation $(5)$ of \cite{Beres}, yet this expression of $d^{\lambda}(t)$ for general $\lambda$ seems to be new. The function $d^{\lambda}(t)$ was studied in \cite{Beres}, when $\left(\lambda_N\left(1^{c}\right)\right)_{ N \in \mathbb{N}}$ is constant and equal to a positive integer $a$: using $t^{\lambda}_c$ defined before, it was shown that $d^{\lambda}(t)$ is $C^{\infty}$ on a subset of the form $\mathbb{R}^{+} \subset I$, with $I$ a bounded interval of $] t_0, \infty[$,  for any $t < t^{\lambda}_c$, $d^{\lambda}(t) = \frac{t}{a}$ and $ \left(d^{\lambda}\right)''\left((t^{\lambda}_{c}\right)^{+}) = - \infty$. Using the Stirling's formula, it is easy to see that for the random walk which only jumps by multiplication by a $k$-cycle, the set $I$ is empty.

\subsection{Preliminary results}
\label{prep}

Before we prove the theorems of Section \ref{generaltheorem}, we present some preliminary results which are not stated in \cite{Gab2}.  

\subsubsection{Exclusive moments for permutation matrices}

Let $k$ be a positive integer. 

\begin{definition}
\label{parure}
Let $I=\{b_1,...,b_{s}\}$ be a partition of $\{1,...,k\}$, let $\sigma$ be an irreducible permutation of $\{1,...,s\}$, and let $i_0 \in \{1,...,k\}$. For any integer $l \in \{1,...,s\}$, we denote by $b'_l$ the set $\{j \in \{1',...,k'\}, \exists  i \in b_l, j=i'\}$. The partition:
\begin{align*}
p_1 = \left\{b_l \cup b'_{\sigma(l)}, l \in \left\{ 1,...,s\right\}\right\}
\end{align*}
is called the {\em necklace} associated with $(I, \sigma)$, and the partition: 
\begin{align*}
p_2 = \left\{b_l \cup b'_{\sigma(l)}, l \in \left\{ 1,...,s\right\}\setminus \{i_0\}\right\}
\end{align*}
is called the {\em chain} associated with $(I, \sigma, i_0)$. The {\em true-length} of $p_1$ and $p_2$, denoted $| p_1|$ and $|p_2|$, is equal to $s$. 

A partition $p$ in $\mathcal{P}_k$ is a {\em parure} if for any cycle $c$ of $p$ the extraction of $p$ on $c$, denoted by $p_c$, is either a chain or a necklace. The true-length of $p$ is: 
\begin{align*}
| p| = \sum_{c \text{ cycle of } p } | p_{c}|. 
\end{align*}
\end{definition}

Let $N$ be a positive integer, let $S$ be a permutation in $\mathfrak{S}(N)$. For any positive integer $l$, $(1,...,l)$ is the $l$-cycle in $\mathfrak{S}_l$. Recall that $m_{p^{c}}$ is the normalized exclusive moment. 

\begin{proposition}
\label{critereconv}
Let $p$ be a partition in $\mathcal{P}_k$. If $p$ is not a parure then $ m_{ p^{c}} \left(S\right)  = 0$. If $p$ is a necklace, then $m_{p^{c}}\left( S\right) = m_{\left(1,...,| p |\right)^{c}}\left( S\right)$. If $p$ is a chain, then: 
\begin{align*}
m_{p^{c}}\left( S\right) = 1- \sum_{l=1}^{ | p|} m_{(1,...,l)^{c}}\left( S\right). 
\end{align*}
\end{proposition}

\begin{proof}
Let $u$ and $v$ two elements of $\{1,...,k\}$  which are in the same block of $p$. Then $u'$ and $v'$ must be in the same block of $p$ if one wants $m_{p^{c}}(S)$ not to be equal to zero. This is a consequence of the fact that for any $i,j,l \in \{1,...,N\}$, 
\begin{align}
\label{equat}
S_{i}^{j}S_{i}^{l} = S_{i}^{j} \delta_{l=j} \text{ and } S_{j}^{i} S^{i}_{l} = S_{j}^{i} \delta_{l=j}. 
\end{align}
The same result holds if one exchanges $k$ and $k'$. Yet, if $p$ is not a parure, these conditions on the blocks of $p$ are not satisfied, thus $m_{p^{c}}\left(S\right) = 0$.

The assertion about the exclusive moments when $p$ is a necklace is a direct consequence of Equations (\ref{equat}). Now, let us suppose that $p$ is a chain. Using Equations (\ref{equat}), we get that $m_{p^{c}}\left( S \right)$ is the fraction of elements of $ \{1,...,N\} $ which period in $S$ is strictly greater than $| p|$. Thus it is equal to one minus the fraction of elements of $\{1,...,N\}$ which period in $S$ is less than $| p|$. Since for any positive integer $l$, $m_{(1,...,l)^{c}}\left( S \right)$ is the fraction of elements of $ \{1,...,N\}$ which period in $S$ is equal to $l$, we get the following equality: 
\begin{align*}
m_{p^{c}}\left( S \right) = 1 - \sum_{l=1}^{| p|} m_{(1,...,l)^{c}} (S),
\end{align*}
which is the equality we had to prove.
\qed \end{proof}

When one considers permutation matrices, an interesting link occurs between moments and exclusive moments. 

\begin{proposition}
\label{lienmomentetexclu}
For any positive integer $k$, 
\begin{align*}
Tr\left[ S^{\otimes k} \circ \rho_N \left((1,...,k)\right) \right] = \sum_{d \in \mathbb{N}^{*}, d | k} Tr\left[ S^{\otimes d}  \circ \rho_N \left( (1,...,d)^{c}\right)\right]. 
\end{align*}
\end{proposition}

\begin{proof}
This is due to the fact that for any positive integer $k$, 
$$Tr\left[ S^{\otimes k} \circ \rho_N \left((1,...,k)\right) \right]= Tr\left(S^{k}\right)$$ is equal to the number of elements $i \in \{1,...,N\}$ whose period divides $k$ and $Tr\left[ S^{\otimes d}  \circ \rho_N \left( (1,...,d)^{c}\right)\right]$ is equal to the number of elements $i \in \{1,...,N\}$ whose period is equal to $d$.
\qed \end{proof}

\subsubsection{Criterion of non $\mathfrak{S}$-freeness}
Let us state some consequence of Lemma \ref{eq:nonlibre} when one considers random matrices whose entries are equal either to $0$ or $1$. Recall that ${\sf 0}_2$ is the partition $\{\{1,2,1',2'\}\}$ in $\mathcal{P}_2$. Let us recall that when, in an observable or in a coordinate number, we do not put the index $N$, it means that we take the limit as $N$ goes to infinity.

\begin{lemma}
\label{calculkappa}
Let $S = \left(S_{N}\right)_{N \in \mathbb{N}}$ be an element of $L^{\infty^{-}}\!\otimes \mathcal{M}(\{0,1\})$. Let us suppose that $S$ converges in $\mathcal{P}$-distribution then: 
\begin{align*}
\mathbb{E}\kappa_{{\sf 0}_{2}} [S] = \mathbb{E}m_{id_1}[S] - \mathbb{E}m_{id_2}[S]. 
\end{align*}
If the asymptotic $\mathcal{P}$-factorization property holds for $S$, then: 
\begin{align*}
\mathbb{E}\kappa_{{\sf 0}_{2}}\left[S\right] = \mathbb{E}m_{id_1}[S] \left(1- \mathbb{E}m_{id_1}[S]\right). 
\end{align*}
\end{lemma}

\begin{proof}
Indeed, we have $\mathbb{E}\kappa_{{\sf 0}_{2}}\left[S\right] =  \mathbb{E}m_{{\sf 0}_{2}}[S] - \mathbb{E}m_{id_2}[S]. $ Yet, for any integer $N$, $S_{N}$ is a matrix of zeros and ones, thus for any positive integer $N$, $ \mathbb{E}m_{{\sf 0}_{2}}\left[S_{N}\right] = \mathbb{E}m_{id_1}\left[S_{N}\right]$. This implies that: 
\begin{align*}
\mathbb{E}\kappa_{{\sf 0}_{2}} [S] = \mathbb{E}m_{id_1}[S] - \mathbb{E}m_{id_2}[S]. 
\end{align*}
The second assertion is a direct consequence of the $\mathcal{P}$-factorization property. 
\qed \end{proof}

This calculation allows us to state the following criterion of non-freeness for $\{0,1\}$-valued random matrices. 

\begin{proposition}
\label{critere}
Let $S_1$ and $S_2$ be two elements of $L^{\infty^{-}}\!\otimes \mathcal{M}(\{0,1\})$ which converge in $\mathcal{P}$-distribution, satisfy the asymptotic $\mathcal{P}$-factorization property and which are asymptotically $\mathcal{P}$-free. If $\mathbb{E}m_{id}[S_1] \notin \{0,1\}$ and $\mathbb{E}m_{id}[S_2] \notin \{0,1\}$, then $S_1$ and $S_2$ are not asymptotically free. 
\end{proposition}

\begin{proof}
This is a consequence of Lemmas \ref{eq:nonlibre} and \ref{calculkappa}.
\qed \end{proof}

\subsubsection{Measures}
The following lemma is a special and easy case of the problem of moments. 
\begin{lemma}
\label{momentprob}
Let $(\kappa_n)_{n \in \mathbb{N}}$ be a sequence of positive numbers such that $\sum_{i=1}^{\infty} \kappa_i \leq \kappa_0$. There exists a unique measure $\mu$ on $\mathbb{U}$ whose weight is equal to $\kappa_0$ such that: 
\begin{align*}
\forall n \in \mathbb{N}^{*}, \int_{\mathbb{U}} z^{n} d\mu(z) = \sum_{d \in \mathbb{N}^{*}, d | n}\kappa_d.
\end{align*}
Besides $\mu$ is given by: 
\begin{align}
\label{eq:form}
\sum_{n \in \mathbb{N}^{*}} \sum_{k=0}^{n-1} \frac{\kappa_{n}}{n} \delta_{e^{\frac{2ik\pi}{n}}} + \left[\kappa_0-\sum_{i=1}^{\infty} \kappa_i\right] \lambda_{\mathbb{U}}, 
\end{align}
where we recall that $\lambda_{\mathbb{U}}$ is the uniform measure on the circle. 
\end{lemma}

\begin{proof}
Any measure on the unit circle $\mathbb{U}$ is characterized by its non-negative moments. It is enough to see that the moments of the measure given by Equation (\ref{eq:form}) are equal to the ones expected. 
\qed \end{proof}

\begin{remark} Let $\mu$ be given by Equation (\ref{eq:form}), the weight of the purely atomic part of $\mu$ is equal to $\sum_{n \geq 1} \kappa_n$. 

\end{remark}

\subsection{Proofs of the theorems stated in Section \ref{generaltheorem}}
\label{proof}
Now that we have gathered all the notions and tools that we need, we can prove the theorems of Section \ref{generaltheorem}.
\begin{proof}[Theorem \ref{convergence1} and Theorem \ref{converlevy}]
For any positive integer $N$, let us consider $\lambda_N$ a conjugacy class of $\mathfrak{S}(N)$. Let us suppose that $\left(\lambda_N\right)_{N \in \mathbb{N}}$ converges as $N$ goes to infinity. For any positive integer $N$, let us consider $\left(S_t^{N}\right)_{t \geq 0}$ a $\lambda_{N}$-random walk on $\mathfrak{S}(N)$. Let $N$ and $k$ be two positive integers and let us define: 
\begin{align*}
G_k^{N} = \frac{d}{dt}_{| t=0} \mathbb{E}\left[\left(S_t^{N}\right)^{\otimes k}\right] = \frac{N}{\lambda_N(1^{c})}\frac{1}{\# \lambda_N} \sum_{\sigma \in \lambda_N} \big[ \sigma^{\otimes k} - id_N^{\otimes k} \big].
\end{align*}

Let $p$ be a partition in $\mathcal{P}_k$ and let $\sigma_N \in \lambda_N$. Let us remark that: 
\begin{align*}
m_{p^{c}} \left(G^{N}_k\right) = \frac{N}{\lambda_N(1^{c})} \left[ m_{p^{c}}\left(\sigma_N^{\otimes k}\right) - m_{p^{c}} \left( id_N^{\otimes k} \right)\right], 
\end{align*}
thus, using Proposition \ref{critereconv}, if $p$ is not a parure, then $m_{p^{c}} \left(G^{N}_k\right)  =0$.

Let us suppose that $p$ is an irreducible parure, then it is either a necklace or a chain. Let us suppose that $p$ is a necklace. Using Proposition \ref{critereconv}, we can suppose that it is a cycle. Let us suppose that $p = (1,...,k)$, then: 
\begin{align*}
m_{p^{c}}\left(G_k^{N}\right) = \frac{N}{\lambda_N(1^{c})} \left[ m_{(1,...,k)^{c}}\left(\sigma_N^{\otimes k}\right) - m_{(1,...,k)^{c}} \left( id_N^{\otimes k} \right)\right]. 
\end{align*}
If $k=1$, then: 
\begin{align*}
m_{p^{c}}\left(G_k^{N}\right) =\frac{N}{\lambda_N(1^{c})} \left[ \frac{\lambda_N(1)}{N} - \frac{N}{N} \right] = -1. 
\end{align*}
If $k \ne 1$, then: 
\begin{align*}
m_{p^{c}}\left(G_k^{N}\right) =\frac{N}{\lambda_N(1^{c})} \left[ \frac{\lambda_N(k)}{N}\right] = \frac{\lambda_N(k)}{\lambda_N(1^{c})} \underset{N \to \infty}{\longrightarrow} \lambda(k). 
\end{align*}
If $p$ is a chain, let us remark that, using again Proposition \ref{critereconv}:
\begin{align*}
m_{p^{c}}\left(G_k^{N}\right)& = \frac{N}{\lambda_N(1^{c})} m_{p^{c}}\left(\sigma_N^{\otimes k}\right) \\& = \frac{N}{\lambda_N(1^{c})} \left[1- \sum_{l=1}^{| p|} m_{(1,...,l)^{c}}\left(\sigma_N^{\otimes l}\right)\right] \underset{N \to \infty}{\longrightarrow} 1-\sum_{l=2}^{| p|} \lambda(l). 
\end{align*}

Thus, for any irreducible partition, $m_{p^{c}}\left(G_k^{N}\right)$ converges as $N$ goes to infinity. Yet, if $p$ is irreducible, any partition $p'$ which is coarser than $p$ is also irreducible. This implies that, for any irreducible partition, $m_{p}\left(G_k^{N}\right)$ converges as $N$ goes to infinity. 

Let us remark that for any partition $p$ in $\mathcal{P}_k$: 
\begin{align}
\label{eq:gsigma}
m_{p}\left(G_k^{N}\right) = \frac{N}{\lambda_N(1^{c})} \left[m_{p} \left(\sigma_N^{\otimes k} \right)-1 \right].
\end{align}
Thus, we have proved that for any irreducible partition $p$, $m_{p}(\sigma_{N}^{\otimes k})$ converges as $N$ goes to infinity. Besides, if $(\lambda_N)_{N \in \mathbb{N}}$ is evanescent, $\frac{\lambda_N(1^{c})}{N}$ goes to infinity: $\lim\limits_{N \to \infty} m_{p}(\sigma_{N}^{\otimes k}) = 1$. 
Let $r$ be a positive integer, let us consider $r$ irreducible partitions $p_1$, ..., $p_r$, we have: 
\begin{align*}
m_{p_1 \otimes...\otimes p_r} \left(G_k^{N} \right) &= \frac{N}{\lambda_N(1^{c})}\left[m_{p_1\otimes...\otimes p_r} \left(\sigma_N^{\otimes k} \right)-1 \right]\\
&= \frac{N}{\lambda_N(1^{c})}\left[\prod_{i=1}^{r}m_{p_i} \left(\sigma_N^{\otimes k} \right)-1 \right]\\
&=\sum_{i=1}^{r} \left( \frac{N}{\lambda_N(1^{c})} \left[m_{p_i} \left(\sigma_N^{\otimes k} \right)-1\right] \prod_{l=i+1}^{r} m_{p_l} \left(\sigma_N^{\otimes k} \right)\right)\\
&=\sum_{i=1}^{r} \left( m_{p_i}\left( G_k^{N}\right)\prod_{l=i+1}^{r} m_{p_l} \left(\sigma_N^{\otimes k} \right)\right).
\end{align*}
This proves that $m_{p_1 \otimes...\otimes p_r} \left(G_k^{N} \right)$ converges as $N$ goes to infinity. Thus for any positive integer $k$, for any partition $p \in \mathcal{P}_k$, $m_{p}\left( G_k^{N} \right)$ converges as $N$ goes to infinity. 

Using Theorem \ref{th:convlevy}, the family $\left( S_t^{N}\right)_{t \geq 0}$ converges in $\mathcal{P}$-distribution toward the $\mathcal{P}$-distribution of a multiplicative $\mathcal{P}$-L\'{e}vy process.

In particular, for any $t \geq 0$, $\left( S_t^{N}\right)_{N \in \mathbb{N}}$ converges in $\mathcal{P}$-distribution as $N$ goes to infinity. Using the discussion after Theorem \ref{th:conv} we deduce that the mean eigenvalue distributions of $\left( S_t^{N}\right)_{N \in \mathbb{N}}$ converges as $N$ goes to infinity to a probability measure $\overline{\mu}_{t}^{\lambda}$ defined on the circle $\mathbb{U}$. Besides, the measure $\overline{\mu}_{t}^{\lambda}$ is characterized by the fact that for any positive integer $n$: 
\begin{align*}
\int_{\mathbb{U}} z^{n} d\overline{\mu}_{t}^{\lambda} = \lim_{N \to \infty} \mathbb{E}m_{(1,...,n)}\left[S_{t}^{N}\right]. 
\end{align*}
Using Proposition \ref{lienmomentetexclu}, we get that: 
\begin{align*}
\int_{\mathbb{U}} z^{n} d\overline{\mu}_{t}^{\lambda} = \sum_{d \in \mathbb{N}^{*}, d| n} \lim_{N \to \infty} \mathbb{E}m_{(1,...,d)^{c}}\left[S_{t}^{N}\right]. 
\end{align*}
We are in the setting of Lemma \ref{momentprob} thus $\overline{\mu}_{t}^{\lambda}$ is equal to: 
\begin{align*}
\overline{\mu}_{t}^{\lambda} = \sum_{n \in \mathbb{N}^{*}} \sum_{k=0}^{n-1} \frac{m_{n^c}(t)}{n}\delta_{e^{\frac{2ik\pi}{n}}} + m_{\infty^{c}}(t) \lambda_{\mathbb{U}}, 
\end{align*}
where for any integer $n \in \mathbb{N}^{*}$, $m_{n^{c}}(t) =  \lim\limits_{N \to \infty} \mathbb{E}m_{(1,...,n)^{c}}\left[S_{t}^{N}\right]$ and $m_{\infty}^{c}(t)$ is such that $\overline{\mu_t}^{\lambda}(\mathbb{U})=1$. 

For any positive integer $k$, any $\sigma\in \mathfrak{S}_{k}$ and any $t \geq 0$, let us denote by $m_{\sigma^{c}}(t)$ the limit $\lim\limits_{N \to \infty} \mathbb{E}m_{\sigma^{c}}\left[S_{t}^{N}\right]$. Using Theorem \ref{th:convergence}, we know that: 
\begin{align*}
m_{\sigma^{c}}(t) = \mathbb{E}\kappa_{\sigma} \left[S_{t}\right]. 
\end{align*}
Using Theorem \ref{th:convlevy} and using the same notations as for this theorem, we get that $m_{\sigma^{c}}(t) $ satisfies the system of equations, $\forall t_0 \geq 0$, $\forall k \in \mathbb{N}^{*}$, $\forall \sigma_0 \in \mathfrak{S}_k$: 
\begin{align*}
\frac{d}{dt}_{| t=t_0} m_{\sigma_{0}^c}(t) = \sum_{p_1 \in \mathcal{P}_k, p_2\in \mathcal{P}_k| p_1 \circ p_2 = \sigma_0, p_1 \prec \sigma_0} \kappa_{p_1}(G_k) \mathbb{E}\kappa_{p_2}[S_{t_0}]. 
\end{align*}
Yet, we saw that if $p_1$ and $p_2$ are two partitions such that $ p_1 \circ p_2 = \sigma_0$ and $p_1 \prec \sigma_0$, then $p_1$ and $p_2$ are two permutations and $p_1 \in [id, \sigma_0]_{\mathcal{P}_k}\cap \mathfrak{S}_k$. Thus $\forall t_0 \geq 0$, $\forall k \in \mathbb{N}^{*}$, $\forall \sigma_0 \in \mathfrak{S}_k$: 
\begin{align*}
\frac{d}{dt}_{| t=t_0} m_{\sigma_{0}^c}(t) = \sum_{\sigma \in \mathfrak{S}_k | \sigma \leq \sigma_0} \kappa_{\sigma}(G_k) m_{^{t}\sigma \sigma_0}(t_0). 
\end{align*}
Using again the fact that for any permutation $\sigma$, $\kappa_{\sigma}[G_k] = \lim\limits_{N \to \infty} m_{\sigma^{c}}\left(G_k^{N}\right)$, we only need to compute $\lim\limits_{N \to \infty} m_{\sigma^{c}}\left(G_k^{N}\right)$ for any permutation in $\mathfrak{S}_k$ in order to finish the proof of Theorem \ref{convergence1}. Let $\sigma_0$ be a permutation in $\mathfrak{S}_k$ and let us compute $m_{\sigma_0^{c}}\left( G_k^{N}\right)$. If $\sigma_0$ is equal to $id_k$, then: 
\begin{align*}
m_{id_k^{c}}\left( G_k^{N}\right) &= \frac{N}{\lambda_N(1^{c})}\frac{1}{\# \lambda_N} \sum_{\sigma \in \lambda_N} \left[ m_{id_k^{c}}\left(\sigma^{\otimes k}\right) - m_{id_k^{c}} \left( id_N^{\otimes k} \right)\right]\\
&=\frac{N}{\lambda_N(1^{c})} \left[ m_{id_k^{c}}\left(\sigma^{\otimes k}\right) - m_{id_k^{c}} \left( id_N^{\otimes k} \right)\right], 
\end{align*}
where $\sigma$ is any permutation in $\lambda_N$. We will use the following convention: for any $n$ and $m$ in $\mathbb{N}$ such that $n-m+1 \leq 0$, $\frac{n!}{(n-m)!} = 0$. If $\sigma$ is a permutation in $\lambda_N$: 
\begin{align*}
 m_{id_k^{c}} \left(\sigma^{\otimes k}\right) = \frac{1}{N^{k}}  \frac{\lambda_N(1)!}{\left(\lambda_N(1)-k\right)!}. 
\end{align*}
Thus: 
\begin{align*}
m_{id_k^{c}}\left( G_k^{N}\right) &=  \frac{N}{\lambda_N(1^{c})} \frac{1}{N^{k}}  \left[\frac{\lambda_N(1)!}{\left(\lambda_N(1)-k\right)!} - \frac{N!}{(N-k)!}\right]\\
&=\frac{N}{\lambda_N(1^{c})}  \left[ \prod_{i=0}^{k-1} \left( 1- \frac{\lambda_N(1^{c}) + i}{N} \right)- \prod_{i=0}^{k-1} \left(1-\frac{i}{N}\right)\right].
\end{align*}
Let us denote by $\alpha$ the limit of $\frac{\lambda_N(1^{c})}{N}$ as $N$ goes to infinity. We get: 
\begin{align*}
\lim_{N \to \infty} m_{id_k^{c}}\left( G_k^{N}\right)  = \left\{
   \begin{array}{ll}
       -k &\text{ if }(\lambda_N)_{N \in \mathbb{N}} \text{ is evanescent, } \\
       \frac{1}{\alpha}( (1-\alpha)^{k} - 1) &\text{ if }(\lambda_N)_{N \in \mathbb{N}} \text{ is macroscopic. }
    \end{array}
\right. 
\end{align*}

Now, let us suppose that $\sigma_0$ is not equal to $id_k$. Let $\sigma$ be in $\lambda_N$, since $m_{\sigma_0}\left(id_N^{\otimes k}\right)= 0$: 
\begin{align*}
m_{\sigma_0^{c}}\left( G_k^{N}\right) = \frac{N}{\lambda_N(1^{c})} m_{\sigma_0^{c}} \left( \sigma^{\otimes k}\right). 
\end{align*}
Let us denote by $[\sigma_0]$ the conjugacy class of $\sigma_0$. As we already saw for $\lambda_N$, we can see $[\sigma_0]$ as a vector $([\sigma_0](i))_{i=1}^{N}$. It is easy to see that: 
\begin{align*}
m_{\sigma_0^{c}} \left( \sigma^{\otimes k}\right) =\frac{1}{N^{{\sf nc}(\sigma_0\vee id)}} \prod_{i=1}^{k} \frac{\left( \frac{\lambda_N(i)}{i}\right)!}{ \left(\frac{\lambda_N(i)}{i} - \frac{[\sigma_0](i)}{i}\right)!} i^{\frac{[\sigma_0](i)}{i}}, 
\end{align*}
thus: 
\begin{align*}
m_{\sigma_0^{c}}\left( G_k^{N}\right) = \frac{N}{\lambda_N(1^{c})} \frac{1}{N^{{\sf nc}(\sigma_0\vee id)}} \prod_{i=1}^{k} \frac{\left( \frac{\lambda_N(i)}{i}\right)!}{ \left(\frac{\lambda_N(i)}{i} - \frac{[\sigma_0](i)}{i}\right)!} i^{\frac{[\sigma_0](i)}{i}}.
\end{align*}
Let us notice that ${\sf nc}(\sigma_0\vee id) = \sum_{i=1}^{k} \frac{[\sigma_0](i)}{i}$. Thus $m_{\sigma_0^{c}}\left( G_k^{N}\right)$ is equal to: 
\begin{align*}
 \frac{N}{\lambda_N(1^{c})}\left( \prod_{i=2}^{k} \frac{1}{N^{\frac{[\sigma_0](i)}{i}}}  \frac{\left( \frac{\lambda_N(i)}{i}\right)!}{ \left(\frac{\lambda_N(i)}{i} - \frac{[\sigma_0](i)}{i}\right)!} i^{\frac{[\sigma_0](i)}{i}}\right) \frac{1}{N^{[\sigma_0](1)}}\frac{\lambda_N(1)!}{ \left( \lambda_{N}(1) - [\sigma_0](1)\right)!}.
\end{align*}

We recall that for any $i \geq 2$, there exists $\lambda(i)$ such that $\frac{\lambda_N(i)}{ \lambda_N(1^{c})}$ converges to $\lambda(i)$ as $N$ goes to infinity, and $\lim_{N \to \infty} \frac{\lambda_N(1)}{N} \to 1-\alpha$. Thus, as $N$ goes to infinity, $m_{\sigma^{c}}\left( G_k^{N}\right)$ has the same limit as: 
\begin{align*}
\frac{N}{\lambda_{N}(1^{c})}\left( \prod_{i=2}^{k} \left( \frac{\lambda_N(i)}{\lambda_N(1^{c})} \frac{\lambda_N(1^{c})}{N}\right)^{\frac{[\sigma_0](i)}{i}} \right)\left(\frac{ \lambda_N(1)}{N}\right)^{[\sigma_0](1)}, 
\end{align*}
or the same limit as: 
\begin{align*}
\left(\frac{\lambda_N(1^{c})}{N}\right)^{{\sf nc}(\sigma_0 \vee id) - [\sigma_0](1) - 1} \prod_{i=2}^{k} \left( \frac{\lambda_N(i)}{\lambda_N(1^{c})}\right)^{\frac{[\sigma_0](i)}{i}}  \left(\frac{ \lambda_N(1)}{N}\right)^{[\sigma_0](1)}.
\end{align*}
This implies that: 
\begin{align*}
\lim_{N \to \infty} m_{\sigma_0^{c}}\left( G_k^{N}\right) = \alpha^{{\sf nc}(\sigma_0 \vee id) - [\sigma_0](1) - 1} \left(\prod_{i=2}^{k} \left(\lambda(i) \right)^{\frac{[\sigma_0](i)}{i}} \right) (1-\alpha)^{[\sigma_0](1)}.
\end{align*}
Let us remark that, since $\sigma_0\ne id_k$, ${\sf nc}(\sigma_0 \vee id_k) - [\sigma_0](1) - 1$ is always non negative. So the following formula has a meaning even if $\alpha = 0$. Using these calculations, we recover the system of differential equations in Theorem \ref{convergence1}. 

At last, let us suppose that $(\lambda_{N})_{N \in \mathbb{N}}$ is evanescent and let us prove that $(S_{t}^{N})_{t \geq 0}$ does not converge toward a free multiplicative L\'{e}vy process. In order to do so, we will prove that the increments of $\left( S_t^{N}\right)_{t \geq 0}$ are not asymptotically free as $N$ goes to infinity. Let $t_1$ and $t_2$ be two positive reals. For any positive integer $N$, let $S_{t_2}^{'N}$ be a random variable which has the same law as $S_{t_2}^{N}$ and which is independent with $S_{t_1}^{N}$. Since $\left( S_t^{N}\right)_{t \geq 0}$ is a L\'{e}vy process, it is enough to prove that $S_{t_1}^{N}$ and $S_{t_2}^{'N}$ are not asymptotically free as $N$ goes to infinity. Using Theorem \ref{th:free}, we already know that  $S_{t_1}^{N}$ and $S_{t_2}^{'N}$ are asymptotically $\mathcal{P}$-free. Besides, for any real $t\geq 0$, $\mathbb{E}m_{id_1}[S_t] = \mathbb{E}m_{id_1^{c}}[S_t]$: using the differential system of equations proved in Theorem \ref{convergence1}, for any $t_0 >0 $, $\mathbb{E}m_{id_1}[S_{t_0}] = e^{-t_0} \notin \{ 0,1\}$. In the following, we will see that the asymptotic $\mathcal{P}$-factorization property holds for $\left(S_{t_1}^{N}\right)_{N \in \mathbb{N}}$ and $\left(S_{t_2}^{'N}\right)_{N \in \mathbb{N}}$: using Proposition \ref{critere}, $S_{t_1}^{N}$ and $S_{t_2}^{'N}$ are not asymptotically free. 
\qed \end{proof}

We have proved the convergence in $\mathcal{P}$-distribution: let us understand when the convergence holds in probability or not. 

\begin{proof}[Theorem \ref{converfac}]
Let us suppose that $\left( \lambda_N \right)$ is macroscopic. The second equation in Theorem \ref{th:convergence} implies that $\mathbb{E}m_{id_2}[S_t] =\mathbb{E}m_{id_2^{c}}[S_t]$. Using this equality and  the system of differential equations satisfied by the limits of the observables, it is easy to see that the family $\left(S_{t}^{N}\right)_{N \in \mathbb{N}}$ does not satisfy the asymptotic $\mathcal{P}$-factorization property: for any $t>0$: 
\begin{align}
\label{eq:nonfact}
\lim\limits_{N \to \infty} \mathbb{E}m_{id_2}\left[S_t^{N}\right] \neq \left(\lim\limits_{N \to \infty} \mathbb{E}m_{id_1}\left[S_t^{N} \right]\right)^{2}. 
\end{align}

Let us suppose that $\left( \lambda_N \right)_{N \in \mathbb{N}}$ is evanescent. Let $p$ be a partition in $\mathcal{P}_k$; we can suppose, up to a permutation of the columns, that there exist $r$ irreducible partitions $p_1$, ..., $p_r$ such that $p$ is equal to $p_1 \otimes...\otimes p_r$. We saw in the proof of Theorem \ref{convergence1} that for any integer $N$:  
\begin{align*}
m_{p}\left(G_k^{N}\right) = \sum_{i=1}^{r}\left( m_{p_i}\left(G_k^{N}\right) \prod_{l=i+1}^{r}m_{p_l} \left( \sigma_N^{\otimes k}\right)\right),
\end{align*}
where $\sigma_N \in \lambda_N$. Using Equation $(\ref{eq:gsigma})$, $\lim\limits_{N \to \infty} m_{p}\left( \sigma_N^{\otimes k}\right) = 1$. Thus, denoting by $m_{p}\left(G_k\right) $ the limit of $m_{p}\left(G_k^{N}\right)$: 
\begin{align*}
m_{p}\left(G_k\right) = \sum_{i=1}^{r}  m_{p_i}\left(G_k\right).
\end{align*}
 The last equation implies that $ \left( G_{k}^{N}\right)_{k,N}$ weakly condensates. By Theorem \ref{th:convlevy}, the process $\left(\left(S_{t}^{N}\right)_{t \geq 0}\right)_{N \in \mathbb{N}}$ satisfies the asymptotic $\mathcal{P}$-factorization property and, using Theorem \ref{th:conv}, it  converges in probability in $\mathcal{P}$-distribution. 

For a positive real $t$, let $\mu_{t}^{\lambda_N}$ the random eigenvalue distribution of $S_t^{N}$.  Since $\left(S_t^{N}\right)_{N \in \mathbb{N}}$ converges in $\mathcal{P}$-distribution, the measures $\mu_{t}^{\lambda_N}$ converge in law to a random measure on $\mathbb{U}$, denoted by $\mu_{t}^{\lambda}$. The measure $\mu_{t}^{\lambda}$ is not random if and only if  $\left(\left(S_{t}^{N}\right)_{t \geq 0}\right)_{N \in \mathbb{N}}$ satisfies the asymptotic $\mathcal{P}$-factorization property: the measure $\mu_{t}^{\lambda}$ is not random if and only if $(\lambda_N)_{N \in \mathbb{N}}$ is evanescent. \qed \end{proof}

From now on, we will suppose that $\left(\lambda_N\right)_{N \in \mathbb{N}}$ is evanescent: the limit of the eigenvalue distributions is non-random. Let us compute this limiting measure. 

\begin{proof}[Theorem \ref{permutations}]
We recall that we already used the following notation: for any positive integer $k$, any $\sigma\in \mathfrak{S}_{k}$ and any $t \geq 0$, we denote by $m_{\sigma^{c}}(t)$ the limit $\lim\limits_{N \to \infty} \mathbb{E}m_{\sigma^{c}}\left[\left(S_{t}^{N}\right)^{\otimes k}\right]$. Besides, we proved that the family $\left( m_{\sigma^{c}}(t)\right)_{t,\sigma}$ satisfies the system of differential equations stated in Theorem \ref{convergence1}. Since we suppose that $(\lambda_N)_{N \in \mathbb{N}}$ is evanescent, for any $t_0\geq 0$ and any $\sigma_0 \in \mathfrak{S}_k$, $\frac{d}{dt}_{| t=t_0} m_{\sigma_{0}^{c}(t)}$ is equal to: 
\begin{align*}
-k m_{\sigma_0^{c}}(t_0) +\!\!\!\!\! \sum_{\sigma \in \mathfrak{S}_{k}\setminus \{id\}, \sigma \leq \sigma_0}  0^{{\sf nc}(\sigma \vee id) - [\sigma](1) - 1}\! \left(\prod_{i=2}^{k} \left(\lambda(i)\right)^{\frac{[\sigma](i)}{i}} \right) m_{(^{t}\sigma \sigma_0)^{c}}(t_0), 
\end{align*}
yet ${\sf nc}(\sigma \vee id) - [\sigma](1) - 1 = 0$ if and only if $\sigma$ is a cycle. Thus, if we set $m_{n^{c}}(t) = m_{(1,...,n)^{c}}(t)$, for any $t_0 \geq 0$ and any positive integer $n$: 
\begin{align}
\label{eq:diff}
\frac{d}{dt}_{| t=t_0} m_{n^{c}(t)}=  -n m_{n^{c}}(t_0) +  \sum_{k=2}^{n} \sum_{\sigma \in \mathfrak{S}_{n},\ \sigma \text{ is a }k\text{-cycle}, \sigma \leq (1,...,n)} \!\!\!\!\!\!\!\!\!\!\!\!\!\!\!\!\!\!\!\! \lambda(k) m_{(^{t}\sigma (1,...,n))^{c}}(t_0). 
\end{align}
Using Theorem \ref{converfac}, the asymptotic $\mathcal{P}$-factorization property holds: we can wite Equation $(\ref{eq:diff})$ only in terms of $(m_{n^{c}}(t))_{n,t}$. For any positive integer $n$, any $t_0\geq 0$: 
\begin{align}
\label{equationdiff1}
\frac{d}{dt}_{| t=t_0} m_{n^{c}}(t) = -n m_{n^{c}}(t) + \sum_{k=2}^{n} \lambda(k)\frac{n}{k} \sum_{(n_1,...,n_k) \in (\mathbb{N}^{*})^{k} | \sum_{i=1}^{k}n_i = n} \prod_{i=1}^{k}  m_{n_{i}^{c}}.
\end{align}
Let us introduce the generating formal series of $\left(e^{nt} m_{n^{c}}(t)\right)_{n \geq 1}$: 
\begin{align*}
{\sf R}(t,z) = \sum_{n\geq 1} e^{nt}m_{n^{c}}(t) z^{n}.  
\end{align*}
Let us remark that ${\sf R}(0,z) = z$. The Equation $\left(\ref{equationdiff1}\right)$ can be written as:
\begin{align}
\label{eqR}
\partial_{t} {\sf R}(t,z) =  z \partial{\sf R}(t,z) {\sf LS}({\sf R})(t,z), 
\end{align} 
where we defined:  
\begin{align*}
{\sf LS}(z) = \sum_{n \geq 1} \lambda(n+1) z^{n}.
\end{align*}
Let us define ${\sf S}(t,z)$ the reciprocal formal series such that for any $t\geq 0$: 
\begin{align*}
{\sf S}(t,{\sf R}(t,z)) = z. 
\end{align*}
Let us remark that ${\sf S}(0,z) = z$. The Equation $(\ref{eqR})$ implies an equation on ${\sf S}$: 
\begin{align*}
\partial_{t}{\sf S}(t,z) = -{\sf LS}(z) {\sf S}(t,z). 
\end{align*}
Thus ${\sf S}(t,z)$ is given by $ {\sf S}(t,z)= z e^{-t {\sf LS}(z)}.$ Let $t \geq 0$ and let $n$ be a positive integer. Using the usual notations, since $e^{nt}m_{n^{c}}(t) = [z^{n}] {\sf R}(t, \bullet)$, we can compute $e^{nt}m_{n^{c}}(t)$ by using the Lagrange inversion. This implies that:

\begin{align*}
[z^{n}] {\sf R}(t,\bullet) = \frac{1}{n}\left[z^{n-1} \right] e^{tn {\sf LS}(z)}, 
\end{align*}
thus $m_{1^{c}}(t) = e^{-t}$ and for $n>1$: 
\begin{align*}
m_{n^{c}}(t) = e^{-nt} \sum_{k=1}^{n-1} t^{k} \frac{n^{k-1}}{k!} \sum_{(i_1,...,i_k) \in \left(\mathbb{N}^{*}\right)^{k}, i_1+...+i_k = n-1} \prod_{j=1}^{k} \lambda(i_j+1), 
\end{align*}
hence the assertions in Theorem \ref{permutations}. 
\qed \end{proof}

Let us prove the assertion on the existence of a phase transition for the random walks on the symmetric group. 

\begin{proof}[Theorem \ref{theo:tcritique}]
Let us suppose that the sequence $\left( \lambda_N\right)_{N\in \mathbb{N}}$ is evanescent. Let us show that the function $f(t) = \sum_{n=1}^{\infty} m_{n^{c}}(t)$, which is equal to $1-m_{\infty^{c}}(t)$ is continuous and converges to $0$ as $t$ goes to infinity. Indeed, we have:
\begin{align*}
f(t) = \sum_{k,n =0}^{\infty} \frac{1}{n} e^{-nt} t^{k} \frac{n^{k}}{k!} p(k,n), 
\end{align*}
where $p(k,n) = \sum_{(i_1,...,i_k) \in \mathbb{N}^{*}, \sum_{j=1}^{k} i_j = n-1} \prod_{j=1}^{k} \lambda(i_j+1)$. For any $k$ and $n$ in $\mathbb{N}$, $f_{k,n}(t)=\frac{1}{n} e^{-nt} t^{k} \frac{n^{k}}{k!} p(k,n)$ is continuous and goes to zero as $t$ goes to infinity, besides $f_{k,n}$ is non-negative and maximal at $t_{k,n} = \frac{k}{n}$ and using Stirling's formula, there exists a constant $C$ such that $f_{k,n}(t_{k,n})=\frac{1}{n}e^{-k}\frac{k^k}{k!} p(k,n) \leq C \frac{1}{k^{3/2}} p(k,n)$. In order to finish, one has to remark that: 
\begin{align*}
\sum_{n \in \mathbb{N}} p(k,n) = \sum_{(i_1,...,i_k) \in \mathbb{N}^{*}} \prod_{j=1}^{k} \lambda(i_j+1) = \left(\sum_{i\in \mathbb{N}^{*}} \lambda(i+1)\right)^{k} \leq 1, 
\end{align*}
thus $\sum_{k,n} f_{k,n}(t_{k,n}) <\infty$. This allows to apply the dominated convergence theorem, thus $f$ is a continous function and converges to zero as $t$ goes to infinity. 

Recall the definition of $t^{\lambda}_c$ given by Equation $(\ref{tcrit})$. Let us prove that $f(t) = 1$ for any $t\leq t_{c}^{\lambda}$ and $f(t)<1$ for any $t >t_{c}^{\lambda}$. Using the generating function ${\sf R}(t,\bullet)$ of $ e^{nt}m_{n^{c}}(t)$, we know that for any real $t \geq 0$: 
\begin{align*}
f(t) = {\sf R}\left(t,e^{-t}\right). 
\end{align*}
Using the fact that ${\sf S}(t,{\sf R}(t,e^{-t})) = e^{-t}$, and given that ${\sf S}(t,z) = ze^{-t{\sf LS}(z)}$, we get that: 
\begin{align*}
{\sf R}(t,e^{-t}) e^{-t{\sf LS}({\sf R}(t,e^{-t}))} = e^{-t}. 
\end{align*}
Thus for any $t \geq 0$, $f(t)$ is a solution in $[0,1]$ of $\Phi_t(z)=ze^{-t({\sf LS}(z)-1)} = 1$. The function $\Phi_t$ is log-concave on $[0,1]$, $\Phi_t(0) = 0$ and $\Phi_t(1) = e^{-(t{\sf LS}(1)-1)}$. If ${\sf LS}(1) = \sum_{i=2}^{\infty}\lambda(i)$ is not equal to one it must be stricly smaller than $1$, thus in this case for any $t>0$, $\Phi_t(1) >1$ and thus there exists a unique solution of $\Phi_t(z)=1$ in $[0,1]$ which is in fact in $]0,1[$. Thus we recover the delta function in Equation $( \ref{tcrit})$. Let us suppose now that $\sum_{i=2}^{\infty}\lambda(i)=1$. Then ${\sf LS}(1) = 1$: thus, since $\Phi_t$ is log-concave, there exists a solution $\nu_{t}$ (which is unique) of $\Phi_t(z)=1$ on $]0,1[$ if and only if $\Phi_t'(1)<0$. Since $\Phi_t'(1)= 1-t{\sf LS}'(1)$, we get that the critical time after which one observes a solution in $[0,1]$ which is different from the trivial solution $1$ is equal to $\frac{1}{{\sf LS}'(1)}$ which is the value of $t_c$ given by Equation $(\ref{tcrit})$. Since $f(t)$ is a continuous function which must converge to zero as $t$ goes to infinity, it must be equal to $1$ if $t \leq t_c$ and then it must be equal to $\nu_t$ if $t >t_c$.  
\qed \end{proof}

Let us finish with the proof of Corollary \ref{KozmaBerestycki}. 

\begin{proof}[Corollary \ref{KozmaBerestycki}]
Let $t$ be a non-negative real number, let $N$ be a positive integer, we have to understand: 
\begin{align*}
\frac{1}{N} d_{\mathfrak{S}(N)}\left(id_{N}, S_{t}^{N}\right) = 1- \frac{{\sf nc}\left( S_{t}^{N} \vee id_N \right)}{N}. 
\end{align*}
Recall that:
\begin{align*}
\frac{1}{N} {\sf nc}\left( S_{t}^{N} \vee id_N \right) = \sum_{k \geq 1} \frac{1}{k}m_{(1,...,k)^{c}} \left( S_{t}^{N}\right)
\end{align*}
since $m_{(1,...,k)^{c}} \left( S_{t}^{N}\right)$ is the fraction of integers in $\{1,...,N\}$ whose period in $S_{t}^{N}$ is equal to $k$. It remains to see if one can interchange the limit and the sum. For any positive integer $N$, for any $\sigma \in \mathfrak{S}(N)$, if $c_{k}(\sigma)$ is the numbers of cycles of size $k$ in $\sigma$, we have for any $K\in \mathbb{N}^{*}$:
\begin{align*}
\sum_{k \geq K} \frac{1}{k} m_{(1,...,k)^{c}} \left( \sigma^{\otimes k}\right) = \frac{1}{N} \sum_{k \geq K} c_{k}(\sigma) \leq \frac{1}{N} \frac{N}{K} = \frac{1}{K},  
\end{align*}
since there can not be more than $\frac{N}{K}$ cycles in $\sigma$ of size bigger than $K$. Thus: 
\begin{align*}
 \sup_{N} \sum_{k \geq K} \frac{1}{k}m_{(1,...,k)^{c}} \left( S_{t}^{N} \right) \underset{K \to \infty}{ \longrightarrow} 0, 
\end{align*}
almost surely. We can interchange limits and, since for any integer $k \geq 1$, $m_{(1,...,k)^{c}} \left( S_{t}^{N} \right)$ converges in probability to $m_{k^{c}}(t)$, we have the convergence in probability: 
\begin{align*}
\lim_{N \to \infty} \frac{1}{N} {\sf nc}\left( S_{t}^{N} \vee id_N \right)= \sum_{k \geq 1}\lim_{N \to \infty}  \frac{1}{k}m_{(1,...,k)^{c}} \left( S_{t}^{N} \right)= \sum_{k=1}^{\infty} \frac{1}{k} m_{k^{c}}(t). 
\end{align*}
This allows us to conclude the proof. 
\qed \end{proof}

\subsection{Log-cumulant calculations}
\label{sec:logcu}
\subsubsection{Brief reminder}
In the article \cite{Gab2}, we studied the log-cumulant invariant of free multiplicative infinitely divisible measures, defined by Equation $(20)$ of \cite{Gab2}. The log-cumulant transform is also an important tool in order to characterize free multiplicative L\'{e}vy processes. Since it could also be the case for more general $\mathcal{P}$-free multiplicative L\'{e}vy processes, we are interested in defining and computing the log-cumulant transform for some examples of multiplicative $\mathcal{P}$-L\'{e}vy processes. The theorem proved in this section is the first computation of the log-cumulants of multiplicative $\mathcal{P}$-L\'{e}vy processes which are not free multiplicative L\'{e}vy processes. 

We can generalize the definition of log-cumulants in the setting of multiplicative $\mathcal{P}$-L\'{e}vy processes: this notion was actually already defined in Definition 3.10 of \cite{Gab1} under an other name. Let $(a_t)_{t \geq 0}$ be a multiplicative $\mathcal{P}$-free L\'{e}vy process which takes values in a $\mathcal{P}$-tracial algebra. The infinitesimal $\boxtimes$-transform of $(a_t)_{t \geq 0}$ is also called the log-cumulant transform of $(a_t)_{t \geq 0}$. This definition can be applied to sequences of matricial L\'{e}vy processes which converge in $\mathcal{P}$-distribution.

For any integer $N$, let $(X^{N}_t)_{t \geq 0}$ be a L\'{e}vy process in the vector space of matrices of size $N$. Let us suppose that $(X^{N}_t)_{t \geq 0}$ converges in $\mathcal{P}$-distribution. In this section, we will always add the asumption that it satisfies the asymptotic $\mathcal{P}$-factorization property. Recall Section 4 of \cite{Gab1} where we defined a convolution $\boxtimes$ on $\left(\bigoplus_{k=0}^{\infty} \mathbb{C}[\mathcal{P}_k]\right)^{*}$. The reader can skip at first the definition of $\boxtimes$ since it will not be needed later. Recall the notion of $\mathcal{R}$-transform explained in Section \ref{sec:remind}: the family $(\mathcal{R}[X_t])_{t \geq 0}$ is a continuous semi-group for the $\boxtimes$-convolution. 

\begin{definition}
The log-cumulant transform of $\left((X^{N}_t)_{t \geq 0}\right)_{N \in \mathbb{N}}$ is the unique element in $ \left(\bigoplus_{k=0}^{\infty} \mathbb{C}[\mathcal{P}_k]\right)^{*}$, denoted by $\mathcal{LR}\left((X_{t})_{t \geq 0}\right)$, such that for any $t_0 \geq 0$:
\begin{align*}
\frac{d}{dt}_{| t=t_0} \mathcal{R}(X_t) = \mathcal{LR}\left((X_{t})_{t \geq 0}\right) \boxtimes \mathcal{R}(X_{t_0}).
\end{align*}
\end{definition}

We have an other characterization of the log-cumulant transform which is a consequence of the results in Section $6.2$ of \cite{Gab2}. 

\begin{lemma}
\label{lem:egaloggen}
Let us consider $G_k^{N} = \frac{d}{dt}_{| t = 0} \mathbb{E}\left[ \left( X_t^{N}\right)^{\otimes k}\right]$ seen as an element of $\mathbb{C}[\mathcal{P}_k(N)]$. Let us suppose that for any integer~$k$, $(G_{k}^{N})_{N}$ converges as $N$ goes to infinity, then for any positive integer $k$, any $p \in \mathcal{P}_k$, $\mathcal{LR}\left((X_t)_{t \geq 0}\right)(p) = \kappa_{p}(G_k)$. 
\end{lemma}

Let $\epsilon_{\boxtimes}$ be the linear form which sends, for any $k \in \mathbb{N}$, $\mathrm{id}_{k}$ on $1$ and any other partition $p$ on $0$. Since we supposed that $(X_t^{N})_{t \geq 0}$ satisfies the asymptotic $\mathcal{P}$-factorization property, the linear form $\mathcal{LR}\left((X_t)_{t \geq 0}\right)$ is a $\boxtimes$-infinitesimal character. This means that for any partitions $p_1$ and $p_2$, 
\begin{align*}
\mathcal{LR}\left((X_t)_{t \geq 0}\right)(p_1\otimes p_2) = \mathcal{LR}\left((X_t)_{t \geq 0}\right)(p_1) \epsilon_{\boxtimes }(p_2) + \epsilon_{\boxtimes}(p_1) \mathcal{LR}\left((X_t)_{t \geq 0}\right)(p_2). 
\end{align*}
This is equivalent to say that: 
\begin{enumerate}
\item for any positive integer $k$, $\mathcal{LR}\left((X_t)_{t \geq 0}\right)(id_k) = k \mathcal{LR}\left((X_t)_{t \geq 0}\right)(id_1)$, 
\item for any partition $p$ which can be written, up to a permutation of the columns, as $p'\otimes id_l$ with ${\sf nc}(p' \vee id) = 1$, $\mathcal{LR}\left((X_t)_{t \geq 0}\right)(p) = \mathcal{LR}\left((X_t)_{t \geq 0}\right)(p')$,
\item for all the other partitions, $\mathcal{LR}\left((X_t)_{t \geq 0}\right)(p) = 0$. 
\end{enumerate}

In the two first cases, we say that $p$ is {\em weakly irreducible}. We will also need a link between the coordinate numbers and the exclusive moments. For that, we need to recall the notion of {\em finer-compatible} which was defined in Definition $2.4$ of \cite{Gab1}. Let $p$ and $p'$ be two partitions in $\mathcal{P}_k$. We say that $p'$ is finer-compatible than $p$ and we denote it by $p' \sqsupset p$ if and only if $p'$ if finer than $p$ and ${\sf nc}(p') - {\sf nc}(p' \vee id) = {\sf nc}(p) - {\sf nc}(p \vee id)$. The following lemma is a consequence of Theorem $5.4$ of \cite{Gab1}. 

\begin{lemma}
\label{lem:tran}
For any positive integer $k$, any $p \in \mathcal{P}_k$, $m_{p^{c}}(G_k) = \sum_{p' \sqsupset p } \kappa_{p'} (G_k). $
\end{lemma}

\subsubsection{The log-cumulants of the limit of random walks on the symmetric group}
For any positive integer $N$, let $\lambda_N$ be a conjugacy of $\mathfrak{S}(N)$ and let us consider $\left(S_t^{N}\right)_{N \in \mathbb{N}}$ a $\lambda_N$-random walk on $\mathfrak{S}(N)$. For any positive integer $t$, let us denote by ${\sf S}^{\lambda}_t$ the family $\left( S_t^{N}\right)_{N \in \mathbb{N}}$. Let us suppose that $\left(\lambda_N\right)_{N \in \mathbb{N}}$ converges as $N$ goes to infinity and that it is evanescent. As we already did, for any $i \geq 2$, we set $\lambda(i) = \lim\limits_{N \to \infty}\frac{\lambda_N(i)}{\lambda_{N}\left( 1^{c}\right)}$ and $\lambda(1)=1$. In the following, we will need the notion of {\em ears}. 

\begin{definition}
Let $k$ be a positive integer, let $i$ be an element of $\{1,...,k\}$ and let $p \in \mathcal{P}_k$. We say that $\{i,i' \}$ is an ear of $p$ if $\{i, i'\}$ are in the same block of $p$. The set of ears of $p$ is denoted by ${\sf E}(p)$. The head of $p$, denoted by ${\sf H}(p)$, is the extraction of $p$ to $\{1,...,k, 1',...,k'\} \setminus \cup_{\{i,i'\} \in {\sf E}(p)} \{i,i'\}$. 
\end{definition}

Let us state the main result about the log-cumulant functional. Recall the notion of true-length defined in Definition \ref{parure}. 

\begin{theorem}
\label{th:logcum}
The log-cumulant transform of $\left({\sf S}^{\lambda}_{t}\right)_{t \geq 0}$, denoted by $\mathcal{LR}^{\lambda}$, is characterized by: 
\begin{enumerate}
\item $\mathcal{LR}^{\lambda}$ is a $\boxtimes$-infinitesimal character, 
\item for any positive integer $k$, for any irreducible partition $p \in \mathcal{P}_k$, if ${\sf H}(p)$ is not a parure then $\mathcal{LR}^{\lambda}(p)=0$, 
\item for any positive integer $k$, for any irreducible partition $p \in \mathcal{P}_k$, if ${\sf H}(p)$ is a necklace then: 
\begin{align*}
\mathcal{LR}^{\lambda}(p)= (-1)^{\# {\sf E}(p)} \lambda( | {\sf H}(p)|), 
\end{align*}
with the convention that $|\emptyset| = 0$ and $\lambda(0)=1$ and if ${\sf H}(p)$ is a chain then:  
\begin{align*}
\mathcal{LR}^{\lambda}(p)= (-1)^{\# {\sf E}(p)} \left( 1-\sum_{i=2}^{| {\sf H}(p) |} \lambda(i)\right). 
\end{align*}
\end{enumerate}
\end{theorem}

\begin{proof}
Since we supposed in this section that $(\lambda_N)_{N \in \mathbb{N}}$ converges and is evanescent, the family of random walks satisfy the asymptotic $\mathcal{P}$-factorization property. We have seen that it implies that $\mathcal{LR}^{\lambda}$ is a $\boxtimes$-infinitesimal character. Using the usual notations and Lemma \ref{lem:egaloggen}, for any partition $p \in \mathcal{P}_k$, $\mathcal{LR}^{\lambda}(p) = \kappa_{p}(G_k)$. Yet we have computed the exclusive moments of the generator in Section \ref{proof}: using Lemma \ref{lem:tran},  for any partition $p$ we known the value of $
\sum_{p' \sqsupset p } \mathcal{LR}^{\lambda}(p')$ which is equal to $m_{p^{c}}(G_k)$. Besides, since $\mathcal{LR}^{\lambda}$ is a $\boxtimes$-infinitesimal character, it is uniquely characterized by the equalites $
\sum_{p' \sqsupset p } \mathcal{LR}^{\lambda}(p') = m_{p^{c}}(G_k)$ for any irreducible partition $p \in \mathcal{P}_k$ and any integer $k$.  Recall that $| p|$ is the true-length of a partition $p$. Using the calculations in the proof of Theorem \ref{convergence1}, for any irreducible partition $p$: 
\begin{enumerate}
\item if $p$ is not a parure then $m_{p^{c}}(G_k)=0$,
\item if $p$ is a necklace of true-length equal to $1$ then $m_{p^{c}}(G_k)=-1$, 
\item if it is a necklace of true-length greater than $1$ then $m_{p^{c}}(G_k) = \lambda(| p|)$, 
\item  if $p$ is a chain then $m_{p^{c}}(G_k) = 1-\sum_{k=2}^{| p|} \lambda(k)$. 
\end{enumerate}
Let us consider the unique $\boxtimes$-infinitesimal character $E$ in $\left(\bigoplus_{k=0}^{\infty} \mathbb{C}[\mathcal{P}_k]\right)^{*}$ which satisfies the conditions $2.$ and $3.$ of the theorem. Let $p$ be an irreducible partition in $\mathcal{P}_k$, it remains to prove that: 
\begin{align*}
m_{p^{c}}(G_k) = \sum_{p' \sqsupset p} E({p'}). 
\end{align*}
Let us recall that $E({p'}) = 0$ if $p'$ is not weakly irreducible. Yet if $p'\sqsupset p$ and $p'$ is weakly irreducible, this means that one can get $p'$ by choosing a certain number of ears of $p$ and by cutting each of them in $p$. Let us consider the two possible cases $p={\sf 0}_k := \{\{1,...,k,1',...,k'\}\}$ or $p\neq {\sf 0}_k$. If $p={\sf 0}_k$, then using the fact that $E(id_k) = -k$: 
\begin{align*}
\sum_{p' \sqsupset {\sf 0}_{k}} E({p'}) &= - \left(\sum_{I \subset \{1,...,k\}, \#I >1} (-1)^{\#I -1} \right) - k  \\
&=- \left(\left(\sum_{l=0}^{k} (-1)^{l-1} \frac{k!}{l!(k-l)!} \right) + 1 \right) = -1 = m_{{\sf 0}_k^{c}}(G_k), 
\end{align*}
the last equality coming from the fact that ${\sf 0}_{k}$ is a necklace of true-length equal to $1$. 

If $p \ne {\sf 0}_k$, then: 
\begin{align*}
\sum_{p' \sqsupset p} E({p'})= \sum_{I \subset {\sf E}(p)} (-1)^{\#{\sf E}(p) - \# I} E({{\sf H}(p)}) = \delta_{\# {\sf E}(p) = 0}E({{\sf H}(p)}) &= \delta_{\# {\sf E}(p) = 0}E({p}) \\&= \delta_{\# {\sf E}(p) = 0} m_{p^{c}}(G_k) \\&= m_{p^{c}}(G_k),
\end{align*}
the last equality coming from the fact that the only irreducible parure in $\mathcal{P}_k$ which has ears is ${\sf 0}_k$. 
\qed \end{proof}

Let us remark that we could have tried to prove the last theorem by computing directly the coordinate numbers of $G_k^{N}$. For example, if one considers the random walk by transpositions:  
\begin{align*}
 G_{k}^{N} &= \frac{1}{N-1}\sum_{\tau \in \mathcal{T}_N} \left(\tau^{\otimes k}-Id_N^{\otimes k} \right) \\&= \frac{1}{2(N-1)}\sum_{i,j=1}^{N} \left( \left(Id_N-E_i^{i}-E_{j}^{j}+E_{i}^{j}+E_{j}^{i} \right)^{\otimes k}  -Id_N^{\otimes k}\right), 
\end{align*}
where $Id_N$ is the identity matrix of size $N$. We used the same notations as in Section \ref{sec:remind} and $\mathcal{T}_N$ is the set of transpositions. One can develop the tensor product and compute the coordinate numbers and their limits. Yet, it becomes less tractable as soon as one considers general random walks.

\section{Large $N$ limit of the $\mathfrak{S}(N)$-Yang-Mills measure}
\label{YMConv}

We will not go into all the details of the theory of planar Yang-Mills fields: one can read \cite{Gabholo} and \cite{Levy1} for an introduction on this subject. Yet, our presentation will be adequate so that the reader does not have to read other articles in order to understand the main result of this section, namely Theorem \ref{masterfield}. The general ideas are all taken from the article \cite{FGA} where asymptotics of unitary Yang-Mills measures are proved. In this article, the Yang-Mills measure with $\mathfrak{S}(N)$ gauge group is the planar Markovian holonomy field associated with the $\mathcal{T}_N$-random walk, where $\mathcal{T}_N$ is the set of transposition in $\mathfrak{S}(N)$. Yet, this section can easily be generalized to planar Markovian holonomy fields associated with any $\lambda_N$-random walk.

The set of paths $P$ in the plane is the set of rectifiable oriented curves drawn in $\mathbb{R}^{2}$ up to increasing reparametrization. The set of loops based at $0$, denoted by $L_0$, is the set of paths $l$ such that the two endpoints of $l$ are $0$. A loop is simple if it does not intersect with itself, except at the endpoints. We will consider $\sf Aff$ and $\sf Aff_0$ respectively the set of piecewise affine paths in $\mathbb{R}^{2}$ and the set of piecewise affine loops based at $0$. We can define two operations on $P$: the {\em concatenation} and the {\em inversion}. Given two paths $p_1$ and $p_2$ such that the starting point of $p_2$ is the arrival point  of $p_1$, it is natural to concatenate $p_1$ and $p_2$ by gluing them at the arrival point  of $p_1$: it defines a new path $p_1p_2$. The inversion of $p_1$, denoted by $p_1^{-1}$, is defined by changing the orientation of $p_1$. 
T. L\'{e}vy defined in \cite{Levy1}, the notion of {\em convergence with fixed endpoints}. For any $p \in P$, $\underline{p}$ denotes the starting point of $p$ and $\overline{p}$ denotes the arrival point of $p$. Let  $(p_n)_{n \in \mathbb{N}}$ be a sequence of paths. The sequence $(p_n)_{n \in \mathbb{N}}$ converges with fixed endpoints if and only if there exists a path $p$ such that for any integer $n$, $p_n$ and $p$ have the same endpoints and: $$| \l(p_n)-\l( p)| +  \inf \underset{t \in [0,1]}{\sup} | p_n(t) - p(t)| \underset{n \to \infty}{\longrightarrow}  0,$$ where the infimum is taken on the parametrizations of the paths $p_n$ and $p$ and where $\l(p )$ is the length of $p$. 

Let $J$ be a subset of $P$, let $G$ be a group. The set of {\em multiplicative functions} $\mathcal{M}ult\left(J,G\right)$ from $J$ to $G$ is the subset of functions $f$ in $G^{J}$ such that for any $p_1, p_2, p_3\in J$ such that $p_1p_2 \in J$ and $p_3^{-1} \in J$, one has: 
\begin{align*}
f(p_1p_2) &=f(p_2) f(p_1),\\
f\left(p_3^{-1}\right) &= f(p_3 )^{-1}.
\end{align*}
 For any $p \in P$, we define $h_p$ or, with an abuse of notation, $h(p)$, as the evaluation on $p$: 
\begin{align*}
h_p: \mathcal{M}ult(J,G) &\to G\\
 h &\mapsto h( p).  
\end{align*}

We are going to define a gauge-invariant measure on the set of multiplicative functions from $J$ to $\mathfrak{S}(N)$. Thus we endow $\mathcal{M}ult(J,\mathfrak{S}(N))$ with the cylinder $\sigma$-field $\mathcal{B}$ which is the trace on $\mathcal{M}ult(J,G)$ of the cylinder $\sigma$-field on $\mathfrak{S}(N)^{J}$. Let us denote by $V$ the set $\{x \in \mathbb{R}^{2}, \exists\ p \in J, x= \underline{p} \text{ or } x= \overline{p}\}$. For any function $j: V \to \mathfrak{S}(N)$ and any $h \in \mathcal{M}ult(J,\mathfrak{S}(N))$, we define $j \bullet h \in \mathcal{M}ult(J,\mathfrak{S}(N))$ such that: 
\begin{align*}
\forall c \in J, (j \bullet h)( c) = j_{\overline{c}}^{-1} h(c) j_{\underline{c}}. 
\end{align*}
A measure $\mu$ on $\mathcal{M}ult\left(J,\mathfrak{S}(N)\right)$ is gauge-invariant if for any measurable function $f$ from $(\mathcal{M}ult(J,\mathfrak{S}(N)), \mathcal{B})$ to $\mathbb{R}$, for any function $j: V \to \mathfrak{S}(N)$: 
\begin{align*}
\int_{\mathcal{M}ult(P,G)} f ( j \bullet h ) d\mu(h) = \int_{\mathcal{M}ult(P,G)} f ( h ) d\mu(h).
\end{align*}

In \cite{FGA}, the author and his co-authors proved a version of the following theorem which is a generalization of Theorem $3.3.1$ proved by T.L\'{e}vy in \cite{Levy1}. The original formulation by T. L\'{e}vy of this theorem is the first part of Theorem \ref{main}. Let us denote by $dx$ the Lebesgue measure on $\mathbb{R}^{2}$. 

\begin{theorem}
\label{main}
Let $\left(\Gamma_N,d_N\right)_{N \in \mathbb{N}}$ be a sequence of complete metric groups such that for any $N\in \mathbb{N}$, translations and inversion are isometries on $\Gamma_N$. For any integer $N$, let $H_N \in \mathcal{M}ult( {\sf Aff}, \Gamma_N)$ be a multiplicative function. Assume that there exists $K_N \geq 0$ such that for any $N \in \mathbb{N}$, for all simple loop $l \in {\sf Aff}$ bounding a disk $D$, the inequality: 
\begin{align}
\label{hold}
d_N(1,H_N(l)) \leq K_N \sqrt{dx(D)}
\end{align}
 holds. Then for each integer $N$, the function $H_N$ admits a unique extension as an element of $\mathcal{M}ult(P, G)$, also denoted by $H_N$, which is continuous for the convergence with fixed endpoints. 

Let $(E,d)$ be a complete metric space. For any integer $N$, let $\psi_N: \Gamma_N\to E$ be a Lipchitz function of Lipchitz norm $|| \psi_{N} ||_{{\sf Lip}}$. Let us assume that the three following conditions hold: 
\begin{enumerate}
\item for any $l \in {\sf Aff}_0$, $\psi_N\left(H_N(l)\right)$ converges to a limit as $N$ goes to infinity, 
\item $\underset{N \in \mathbb{N}}{\sup} || \psi_N ||_{{\sf Lip}} < \infty$, 
\item $\underset{N \in \mathbb{N}}{\sup} K_N <\infty,$ 
\end{enumerate}
then for any $l \in L_0$, $\big(\psi_N(H_N(l))\big)_{N \in \mathbb{N}}$ converges to a limit $\phi(l)$. Besides, the function: 
\begin{align*}
\phi: L_0&\to E\\
l &\mapsto \phi(l)
\end{align*}
is continuous for the convergence with fixed endpoints. 
\end{theorem}

Recall that $\mathcal{T}_N$ is the set of transpositions in $\mathfrak{S}(N)$. Let $\left(S_{t}^{N}\right)_{ t \geq 0}$ be a $\mathcal{T}_N$ random walk on $\mathfrak{S}(N)$.  Let us explain how the first part of Theorem \ref{main} allows us to construct the Yang-Mills field associated with  $\left(S_{t}^{N}\right)_{ t \geq 0}$.  In order to do so, we need the notion of finite planar graph: it will be the usual notion, except that we ask that the bounded faces are homeomorphic to an open disk. Let $\mathbb{G}$ be a finite planar graph: the set of bounded of faces of $\mathbb{G}$ is denoted by $\mathbb{F}$. For any finite planar graph $\mathbb{G}$, we define $P(\mathbb{G})$ as the set of paths that one can draw by concatenating edges of $\mathbb{G}$. Let us define also $\mathcal{G}({\sf Aff})$ the set of finite planar graphs $\mathbb{G}$ whose edges are piecewise affine.

In order to construct a measure on $\left(\mathcal{M}ult(P,\mathfrak{S}(N)),\mathcal{B}\right)$, first we construct for any $\mathbb{G} \in \mathcal{G}({\sf Aff})$ an associated measure $\mu_{\mathbb{G}}$ on $\big(\mathcal{M}ult(P(\mathbb{G}),\mathfrak{S}(N)),\mathcal{B}\big)$. We will give the construction given by the author in \cite{Gabholo}, but one can have a look at \cite{Levy1} where a different formulation is given. 

We need to introduce the loop paradigm for two dimensional Yang-Mills fields. Let us consider a finite planar graph $\mathbb{G}$ in $\mathcal{G}({\sf Aff})$, let us consider $v_0$ a vertex of $\mathbb{G}$ and $T$ a covering tree of $\mathbb{G}$. Let us consider for any bounded face $F$ of $\mathbb{G}$ a loop $c_F \in P(\mathbb{G})$ which represents the anti-clockwise-oriented boundary $\partial F$. For any vertex $v$ of $\mathbb{G}$, we denote by $[v_0,v]_{T}$ the unique injective path in $T$ which goes from $v_0$ to $v$. Let $L_{v_0}(\mathbb{G})$ be the set of loops $l$ in $P(\mathbb{G})$ such that $ \underline{l} = v_0$. We define the facial lasso ${\sf l}_{ F} \in L_{v_0}(\mathbb{G})$ by: 
$${\sf l}_{F} = [v_0,v]_{T}\ \!c_{F} \ \![v_0,v]_{T} ^{-1}.$$ 
It was proved in Proposition $5.12$ of \cite{Gabholo} that the application: 
\begin{align*}
\Phi_{T,(c_{\mathbb{F}})_{F \in \mathbb{F}}}: \mathcal{M}ult\left( L_{v_0}(\mathbb{G}) , \mathfrak{S}(N) \right)& \to  \left(\mathfrak{S}(N)\right)^{\mathbb{F}} \\
h &\mapsto \left( h \left({\sf l}_{ F} \right)\right)_{F \in \mathbb{F}}, 
\end{align*}
is a bijection and for any loop $l \in L_{v_0}(\mathbb{G})$, there exists a word $w$ in the letters $({\sf l}_F)_{F \in \mathbb{F}}$ and $\left({\sf l}_F^{-1}\right)_{F \in \mathbb{F}}$ such that $h_{l} = w\left( \left(h_{{\sf l}_F}\right)_{F \in \mathbb{F}},  \left(h_{{\sf l}_F^{-1}}\right)_{F \in \mathbb{F}} \right).$

Using Proposition $7.2$ proved by the author in \cite{Gabholo}, we have the following proposition.

\begin{proposition}
\label{creation}
There exists a unique gauge-invariant measure $\mu_{\mathbb{G}}^{v_0,T,(c_F)_{F \in \mathbb{F}}}$ on $\mathcal{M}ult\left( P(\mathbb{G}), \mathfrak{S}(N)\right)$ such that under this measure: 
\begin{enumerate}
\item the random variables $\left( h\left({\sf l}_{ F}\right)\right)_{F \in \mathbb{F}}$ are independent, 
\item for any $F \in \mathbb{F}$, $h\left({\sf l}_{ F}\right)$ has the same law as $S_{dx(F)}^{N}$. 
\end{enumerate}
This measure does not depend neither on the choice of $v_0$ nor $T$ nor on the choice of $\left(c_{F}\right)_{F \in \mathbb{F}}$: we denote it $\mu_{\mathbb{G}}$. 
\end{proposition}

Let $\mathbb{G}$ and $\mathbb{G}'$ be two finite planar graphs in $\mathcal{G}({\sf Aff})$ such that $\mathbb{G}'$ is coarser than $\mathbb{G}$: this means that $P(\mathbb{G}') \subset P(\mathbb{G})$. Any function in $\mathcal{M}ult(P(\mathbb{G}),\mathfrak{S}(N))$ allows us to define, by restriction, an element of $\mathcal{M}ult(P(\mathbb{G}'),\mathfrak{S}(N))$. The measures $\left( \mu_{\mathbb{G}} \right)_{\mathbb{G}}$ are compatible with the applications of restriction: the family of measures $\big(\mathcal{M}ult(P(\mathbb{G}),\mathfrak{S}(N)), \mu_{\mathbb{G}} \big)_{\mathbb{G} \in \mathcal{G}({\sf Aff})}$ is a projective system and, as explained in Proposition $1.22$ of \cite{Gabholo} and in \cite{Levy1}, we can take the projective limit.

\begin{definition}
The affine Yang-Mills measure associated to $\left(S_{t}^{N}\right)_{t \geq 0}$, denoted by $YM_{\sf Aff}^{\mathfrak{S}(N)}$, is the projective limit of: 
\begin{align*}
\big(\mathcal{M}ult(P(\mathbb{G}),\mathfrak{S}(N)), \mu_{\mathbb{G}} \big)_{\mathbb{G} \in \mathcal{G}({\sf Aff})}.
\end{align*}
It is a gauge-invariant measure on $\mathcal{M}ult({\sf Aff}, \mathfrak{S}(N))$. 
\end{definition}

Let us consider a simple loop $l$ in ${\sf Aff}$ and let $\mathbb{G}^{l}$ be the finite planar graph in $\mathcal{G}({\sf Aff})$ which has $l$ as unique edge. In this case, $\mathcal{M}ult(P(\mathbb{G}^{l}), \mathfrak{S}(N))\simeq \mathfrak{S}(N)$ and for any continuous function $f: \mathfrak{S}(N) \to \mathbb{R}$: 
\begin{align}
\label{egalit}
YM_{\sf Aff}^{\mathfrak{S}(N)}\left[f(h_l)\right] = \mathbb{E}\left[f\left(S^{N}_{dx({\sf Int}(l))}\right)\right], 
\end{align}
where ${\sf Int}(l)$ is the bounded component of $\mathbb{R}^{2}\setminus l$. 
This last equality shows that under $YM_{\sf Aff}^{\mathfrak{S}(N)}$, $h_{l}$ has the same law as $S^{N}_{dx({\sf Int}(l))}$. This will allow us to use the first part of Theorem \ref{main} in order to construct the Yang-Mills measure, as it was done by T. L\'{e}vy in \cite{Levy1} and then by the author in \cite{Gabholo}. We need some estimates on the walk $\left(S^{N}_{t}\right)_{t \geq 0}$: in order to do so, let us define a distance on $\mathfrak{S}(N)$ by considering any element of $\mathfrak{S}(N)$ as a permutation matrix of size $N$.

\begin{definition}
\label{distance}
For any $\sigma, \sigma' \in \mathfrak{S}(N)$: 
\begin{align*}
d_{N}(\sigma,\sigma') = \left[2\left(1-\frac{1}{N}Tr\!\left(\sigma \sigma'^{-1}\right)\right)\right]^{\frac{1}{2}}. 
\end{align*}
\end{definition}
Since the permutation matrices are orthogonal, for any $\sigma$ and $\sigma'$ in $\mathfrak{S}(N)$: 
\begin{align*}
d_{N} \left(\sigma,\sigma'\right) = \left[\frac{1}{N}Tr\left(\left(\sigma-\sigma'\right)\ ^{t}\!\left(\sigma-\sigma'\right)\right)\right]^{\frac{1}{2}}. 
\end{align*}
This shows that $d_N$ is a distance on $\mathfrak{S}(N)$. Let us control the distance of $\left(S^{N}_t\right)_{t \geq 0}$ to the identity.

\begin{lemma}
\label{estime}
For any real $t\geq 0$, $\mathbb{E}\left[d_{N}\left(id_N,S_{t}^{N}\right)\right] \leq  \sqrt{2t}. $
\end{lemma}

\begin{proof}
Let $t$ be a non-negative real. By definition: 
\begin{align*}
\mathbb{E}\left[d_{N}\left(id_N,S_{t}^{N}\right)^{2}\right] = 2\left[ 1-\mathbb{E}\left[\frac{1}{N}Tr\left(S_{t}^{N}\right)\right]\right]. 
\end{align*} 
A simple calculation allows us to see that $\frac{1}{N}\sum_{\tau \in \mathcal{T}_N} (\tau-id_N)$, seen as a matrix of size $N$, is equal to $\rho_N \left[\frac{1}{N}{\sf e}_1 -id\right],$ where we recall that $\rho_{N}$ is the representation defined in Section \ref{sec:remind} and ${\sf e}_1$ is the partition $\{\{1\} ,\{1'\}\}$. This implies that for any $t_0\geq 0$:  
\begin{align*}
\frac{d}{dt}_{| t=t_0}\mathbb{E}\left[S_{t}^{N}\right] =\rho_N\left(\frac{1}{N}{\sf e}_1 - id_1\right)\mathbb{E}\left[S_{t_0}^{N}\right]. 
\end{align*}
Thus, by linearity: 
\begin{align*}
\frac{d}{dt}_{| t=t_0} \mathbb{E}\left[ \frac{1}{N} Tr\left(S_{t}^{N}\right)\right] = \frac{1}{N} \mathbb{E}\left[ \frac{1}{N}Tr\left(\rho_N({\sf e}_1) S_{t_0}^{N}\right)\right] -  \mathbb{E}\left[ \frac{1}{N}Tr\left( S_{t_0}^{N}\right)\right], 
\end{align*}
and, using the fact that $\frac{1}{N}Tr\left( \rho_N({\sf e}_1) \sigma \right) = 1$ for any $\sigma \in \mathfrak{S}(N)$, we get the differential equation:
\begin{align*}
\frac{d}{dt}_{| t=t_0} \mathbb{E}\left[ \frac{1}{N} Tr\left(S_{t_0}^{N}\right)\right] &= \frac{1}{N}  -  \mathbb{E}\left[ \frac{1}{N}Tr\left( S_{t_0}^{N}\right)\right], \\
\mathbb{E}\left[ \frac{1}{N} Tr\left(S_{0}^{N}\right)\right] &= 1. 
\end{align*}
The solution is given by the function $t \mapsto \frac{1}{N} + \left(1-\frac{1}{N}\right)e^{-t}$: for any real $t \geq 0$: 
\begin{align}
\label{pointfixe}
\mathbb{E}\left[ \frac{1}{N} Tr\left(S_{t}^{N}\right)\right] = \frac{1}{N}+\left(1-\frac{1}{N}\right) e^{-t}, 
\end{align}
and thus:  
\begin{align*}
\mathbb{E}\left[d_{N}\left(id_N, S_{t}^{N}\right)^{2}\right] = 2\left[1- \frac{1}{N}\right] [1-e^{-t}]. 
\end{align*}
This implies that for any $t\geq 0$ and any positive integer $N$, 
$$\left(\mathbb{E}\left[d_N\left(id_N, S^{N}_t\right)\right]\right)^{2} \leq \mathbb{E}\left[d_{N}\left(id_N, S^{N}_{t}\right)^{2}\right]\leq 2t.$$
This allows us to finish the proof. 
\qed \end{proof}

Using Lemma \ref{estime}, we can prove the following theorem. 
\begin{theorem}
\label{yangmills}
The measure $YM^{\mathfrak{S}(N)}_{\sf Aff}$ can be extended by continuity to a measure on $\mathcal{M}ult(P,\mathfrak{S}(N))$. This means that there exists a measure  $YM^{\mathfrak{S}(N)}$ on $\mathcal{M}ult(P, \mathfrak{S}(N))$ such that: 
\begin{enumerate}
\item the restriction of $YM^{\mathfrak{S}(N)}$ on $\mathcal{M}ult(\sf Aff, \mathfrak{S}(N))$ is equal to  $YM^{\mathfrak{S}(N)}_{\sf Aff}$,
\item for any sequence of paths $(p_n)_{n \in \mathbb{N}}$ and any path $p\in P$ such that $(p_n)_{n \in \mathbb{N}}$ converges with fixed endpoints to $p$, we have: 
\begin{align*}
YM^{\mathfrak{S}(N)}\left[d_{N}\left(h_{p_n}, h_p\right)\right] \underset{n \to \infty}{\longrightarrow}  0. 
\end{align*} 
\end{enumerate}
\end{theorem}

We only recall the proof given in \cite{Levy1}. 
\begin{proof}
Let $(\Omega,\mathcal{A}, \mathbb{P})$ be equal to $\left(\mathcal{M}ult({\sf Aff}, \mathfrak{S}(N)), \mathcal{B}, YM^{\mathfrak{S}(N)}_{\sf Aff}\right)$. For any path $p\in{\sf Aff}$, $h_p$ is a function on $\mathcal{M}ult({\sf Aff}, \mathfrak{S}(N))$ thus it can be seen as a $G$-valued random variable on $(\Omega, \mathcal{A}, \mathbb{P})$.
For any positive integer $N$, let $\Gamma_N = L(\Omega,\mathcal{A},\mathbb{P};\mathfrak{S}(N))$ be the set of $\mathfrak{S}(N)$-valued random variables defined on $\Omega$: this is a group for the pointwise multiplication of random variables. We endow $\Gamma_{N}$ with the distance: 
\begin{align*}
\overline{d_N}(X,Y) = \mathbb{E}\left[d_{N}(X,Y)\right]. 
\end{align*}
It is a distance which is invariant by translations and inversion. 
Let us consider the mapping: 
\begin{align*}
H_N: \sf Aff &\to \Gamma_N\\
l &\to h_l, 
\end{align*}
whih is a multiplicative function. Using Lemma \ref{estime} and Equality (\ref{egalit}),  we get that for any simple loop $l$: $\overline{d_N}(1,h_{l}) \leq\sqrt{2}\sqrt{dx({\sf Int}(l))}.$ We can apply Theorem \ref{main}: there exists an extension: 
\begin{align*}
H_N: P &\to \Gamma_N\\
p &\mapsto H(p)
\end{align*}
which is continuous for the convergence with fixed endpoints: for any sequence of paths $(p_n)_{n \in \mathbb{N}}$ and any path $p\in P$ such that $(p_n)_{n \in \mathbb{N}}$ converges with fixed endpoints to $p$, we get: 
\begin{align*}
\overline{d_N} \left(H_N\left(p_n\right), H_N\left(p \right)\right) = YM_{\sf Aff}^{\mathfrak{S}(N)}\big[d_{N}(H_N(p_n), H_N(p ))\big] \underset{n \to \infty}{\longrightarrow}  0.
\end{align*}
Thus we have constructed a $\mathfrak{S}(N)$-valued process $\left(H_N( p)\right)_{p \in P}$ on $\Omega$ which is stochastically continuous in law and such that for any $p$ and $p'$ in $P$ such that $\overline{p} = \underline{p'}$, almost surely $H_N(p p') =H_N(p') H_N( p)$, $H_N(p^{-1})=H_N( p)^{-1}$. Using Proposition $1.22$ in \cite{Gabholo} this allows us to construct a measure on $\mathcal{M}ul(P, \mathfrak{S}(N))$ called $YM^{\mathfrak{S}(N)}$, such that the process $(h_p)_{p \in P}$ has the same law under $YM^{\mathfrak{S}(N)}$ as the process $(H_N(p ))_{p \in P}$ under $YM_{\sf Aff}^{\mathfrak{S}(N)}$. The measure $YM^{\mathfrak{S}(N)}$ satisfies the desired properties. 
\qed \end{proof}

Now that we have defined the Yang-Mills measure $YM^{\mathfrak{S}(N)}$ for any positive integer $N$, we are interested in the convergence of the Wilson loops under these measures as $N$ goes to infinity.

\begin{definition}
Let $l_0$ be a loop based at $0$, the Wilson loop on $l_0$ is the function: 
\begin{align*}
W_{l_0}^{N}: \mathcal{M}ult(P , \mathfrak{S}(N)) &\to \mathbb{R}\\
 (h_p)_{p\in P} &\mapsto \frac{1}{N} Tr(h_{l_0}).
\end{align*}
seen as a random variable on $(\mathcal{M}ult(P , \mathfrak{S}(N)) , \mathcal{B}, YM^{\mathfrak{S}(N)})$.
\end{definition}

The main result about the limit of Yang-Mills measure on the symmetric group is given by the following result. 
\begin{theorem}
\label{masterfield}
For any loop $l$ based at $0$ the Wilson loop on $l$ converges in probability to a constant denoted by $\phi(l)$ as $N$ goes to infinity. The function: 
\begin{align*}
\phi: L_0 &\to \mathbb{R} \\
l &\mapsto \phi_l, 
\end{align*}
is continuous for the convergence with fixed endpoints. 

The asymptotic factorization property holds: for any positive integer $k$, any $k$-tuple of loops $l_1$, ..., $l_k$ in $L_0$: 
\begin{align*}
YM^{\mathfrak{S}(N)}\left[W_{l_1}^{N}...W_{l_k}^{N}\right] \underset{N \to \infty}{\longrightarrow} \phi(l_1)...\phi(l_k). 
\end{align*}
\end{theorem}
The function $\phi$ in Theorem \ref{masterfield} is called the $\mathfrak{S}(\infty)$-master field. 

\begin{remark}
The vector space $\mathbb{C}[L_0]$ can be endowed with a structure of algebra with the concatenation of loops and can be endowed with a $*$-operation defined by: 
\begin{align*}
(\lambda l)^{*} = \overline{\lambda} l^{-1}.
\end{align*}
The function $\phi$ can be extended by linearity on $\mathbb{C}[L_0]$ and it satisfies some interesting properties: 
\begin{itemize}
\item $\phi$ is continuous for the convergence with fixed endpoints, 
\item $\phi$ is invariant by the homeomorphisms which preserves the Lebesgue measure, 
\item $\phi$ defines an application on $\mathbb{C}[L_0]$: 
\begin{align*}
\Phi : (l,l') \mapsto \phi(l l'^{*}).
\end{align*}
which satisfies the conjugate symmetry, the linearity in the first argument and the positiveness properties: for any $(a_i)_{i=1}^{k} \in \mathbb{C}$ and $(l_i)_{i=1}^{k}$, $\sum_{i,j=1}^{k} a_i\overline{a_j}\phi(l_il_j^{*})$ is non negative.
\end{itemize}
\end{remark}

Let us prove Theorem \ref{masterfield} when one considers only piecewise affine loops. 

\begin{proposition}
\label{convergenceaffine}
For any loop $l$ in $\sf Aff_0$ the Wilson loop $W_{l}^{N}$ converges in expectation and in probability as $N$ tends to infinity to a constant $\phi(l)$. 
\end{proposition}

\begin{proof}
Let $l_0$  be a loop in $\sf Aff_0$. Let $\mathbb{G}$ be a graph in $\mathcal{G}({\sf Aff})$ such that $l_0$ is a loop in $\mathbb{G}$. Let us consider $T$ a covering tree of $\mathbb{G}$, let us consider for any bounded face $F$ of $\mathbb{G}$ a loop $c_F \in P(\mathbb{G})$ which represents $\partial F$ in the anti-clockwise orientation and let us consider the facial lassos ${\sf l}_{F}$ associated with these choices of tree and loops.  Let $w\left(\left({\sf l}_{F}\right)_{F \in \mathbb{F}}, \left({\sf l}_{F}^{-1}\right)_{F \in \mathbb{F}}\right)$ be a word in the letters $({\sf l}_F)_{F \in \mathbb{F}}$ and $\left({\sf l}_F^{-1}\right)_{F \in \mathbb{F}}$ such that $h_{l_0} = w\left( \left(h_{{\sf l}_F}\right)_{F \in \mathbb{F}},  \left(h_{{\sf l}_F^{-1}}\right)_{F \in \mathbb{F}} \right)$.

Using Proposition \ref{creation}, the random variables $(h_{{\sf l}_F})_{F \in \mathbb{F}}$ defined on the probability space $\left(\mathcal{M}ult(L_0,\mathbb{G}),\mathcal{B}, YM^{\mathfrak{S}(N)}\right)$ are independent and for any $F\in \mathbb{F}$, $h_{{\sf l}_F}$ has the same law as $S^{N}_{dx(F)}$. For all positive integer $N$, let $\left(S_{t,N}^{(1)}\right)_{t \geq 0}$, ..., $\left(S_{t,N}^{(\#\mathbb{F})}\right)_{t \geq 0}$ be $\#\mathbb{F}$ independent random walks identically distributed as $(S_{t}^{N})_{t \geq 0}$. The discussion we just had implies that there exist $i_1,...,i_k,j_1,...,j_{k'}$ in $\{1,...,\#\mathbb{F}\}$ and $\sigma$ a $k+k'$-cycle such that the Wilson loop $W_l^{N}$ is equal to: 
\begin{align*}
m_{\sigma}\left( S_{t,N}^{(i_1)}, ..., S_{t,N}^{(i_k)}, \left(S_{t,N}^{(j_1)}\right)^{-1}, ..., \left(S_{t,N}^{(j_{k'})}\right)^{-1} \right), 
\end{align*}
where we recall that $m_{\sigma}$ are the observables defined in Section \ref{sec:remind}. The Theorem \ref{convergencefamille}, applied to the evanescent family $\left((N-2,2,0,...,0)\right)_{N \in \mathbb{N}}$ allows us to conclude. 
\qed \end{proof}

In order to generalize Proposition \ref{convergenceaffine} and to prove Theorem \ref{masterfield}, we need an estimate on the Lipschitz norm of the function which gives the fraction of fixed points of a permutation.  Recall Definition \ref{distance} where we defined a distance $d_N$ on $\mathfrak{S}(N)$.

\begin{lemma}
\label{lip}
For any positive integer $N$ and any permutations $\sigma$ and $\sigma'$ in $\mathfrak{S}(N)$, one has: 
\begin{align*}
\frac{1}{N}\!\left\lvert Tr(\sigma)-Tr(\sigma')\right\rvert\ \leq d_{N}(\sigma,\sigma'). 
\end{align*}
\end{lemma}
\begin{proof}
It is a consequence of the Cauchy-Schwarz's inequality:  
\begin{align*}
\frac{1}{N}\!\left\lvert Tr(\sigma)-Tr(\sigma')\right\rvert\!=\!\frac{1}{N} \left\lvert Tr(\sigma-\sigma')\right\rvert\!\leq\!\!\left[\frac{1}{N} Tr\!\!\left((\sigma\!-\!\sigma')\!\!\text{ }^{t}(\sigma\!-\!\sigma')\right)\right]^{\frac{1}{2}}\!\!\!=\! d_{N}(\sigma,\sigma'), 
\end{align*}
hence the result. 
\qed \end{proof}

We can finish the proof of Theorem \ref{masterfield}. 
\begin{proof}[Theorem \ref{masterfield}]
We will use the second part of Theorem \ref{main}. For this, we consider $\left(\Omega, \mathcal{A}, \mathbb{P}\right)$ a probability space on which is defined for each positive integer $N$ a process $\left(h_{p}^{N}\right)_{p \in P}$ whose law is the law of the canonical process $\left(h_{p}\right)_{p \in P}$ under the $\mathfrak{S}(N)$-Yang-Mills measure associated with the $\mathcal{T}_N$-random walk on $\mathfrak{S}(N)$. Recall the notations defined in the proof of Theorem \ref{yangmills}: we consider $\Gamma_N = L\left(\Omega, \mathcal{A}, \mathbb{P}; \mathfrak{S}(N)\right)$ endowed with the distance $\overline{d_N}$ and we consider the mappings $H_N$ defined from $P$ to $\Gamma_N$. 

Let us denote by $E$ the space $L\left(\Omega, \mathcal{A}, \mathbb{P}; \mathbb{R}\right)$ of real valued random variables defined on $\Omega$. Let us endow $E$ with the distance $d\left(X,Y \right)= \mathbb{E}\left[|X-Y|\right]$. For any positive integer $N$, let: 
\begin{align*}
\psi_N: \Gamma_N &\to E \\
S &\mapsto \frac{1}{N}Tr\left(S\right). 
\end{align*}

Using Lemma \ref{lip}, for any positive integer $N$, $\psi_N$ is a Lipschitz function and $\underset{N \in \mathbb{N}}{\sup} || \psi_N||_{{\sf Lip}} \leq 1$. Besides, using Proposition \ref{convergenceaffine}, and using the dominated convergence theorem, we know that for any $l \in {\sf Aff}_0$, $\psi_N(H_N(l))$ converges in $E$ to a limit which is the non-random variable $\phi(l)$. At last, Lemma \ref{estime} shows that the constant $K_N$ in (\ref{hold}) can be taken equal to $\sqrt{2}$ for any positive integer $N$. Thus we can apply the second part of Theorem \ref{main}: for any $l \in L_0$, $\psi_N \left(H_N(l) \right)= W_l^{N}$ converges in $E$ to a limit $\phi(l)$ and the function:
 \begin{align*}
\phi: L_0 &\to E \\
l &\mapsto \phi(l) 
\end{align*}
is continuous for the convergence with fixed endpoints. 

Let $l$ be a loop based at $0$: one can approximate $l \in L_0$ by a sequence of loops $\left(l_n\right)_{n \in \mathbb{N}}$ in ${\sf Aff}_0$. Since for any positive integer $n$, $\phi(l_n)$ is almost surely constant then $\phi(l)$ is almost surely constant. Besides, the convergence of $W_l^{N}$ to $\phi(l)$ in probability holds since we proved that: 
\begin{align*}
\overline{d_N}(W_l^{N}, \phi(l)) = \mathbb{E}\left[ | W_{l}^{N}-\phi(l)|\right] \underset{N \to \infty}{\longrightarrow} 0. 
\end{align*}
The asymptotic factorization property is a simple consequence of the dominated convergence theorem. 
\qed \end{proof}

\begin{remark}
Theorem \ref{masterfield} is also true for more general sequences of Yang-Mills measures, the proof follows exactly the same steps.
\end{remark}

\section{Random ramified coverings}
\label{sec:cov}
In this section, we present a natural model of random ramified coverings on the unit disk $\mathbb{D}$. This model was first defined in \cite{Levy1}, in Chapter 5 in the general setting of ramified $G$-bundles when $G$ is a finite group. We translate the results for random ramified coverings without any conditions on the monodromy on the boundary. This needs some simple verifications which will not be further discussed here. Let $Y$ be a finite subset of $\mathbb{D}\setminus \partial \mathbb{D}$. 

\begin{definition}
A ramified covering of the disk with ramification locus $Y$ is a continuous mapping $\pi: R \to \mathbb{D}$ from a surface $R$ such that the following conditions hold: 
\begin{enumerate}
\item the restriction of $\pi$ to $\pi^{-1}\left(\mathbb{D}\setminus Y\right)$ is a covering,
\item for all $y \in Y$ and any $p \in \pi^{-1}(y)$, one can find a neighborhood $U$ of $p$ and an integer $n \geq 1$ such that the mapping: 
\begin{align*}
\pi_{| U}: (U,p) &\to (\pi(U), y)\\
x &\mapsto \pi(x)
\end{align*}
is conjugated to the mapping $z \mapsto z^{n}$: $(\mathbb{C},0) \to (\mathbb{C},0).$
\end{enumerate}
The integer $n$ is the order of ramification of $p$ and will be denoted by ${\sf or}( p)$. 

Let $N$ be a positive integer. A ramified covering $\pi: R \to \mathbb{D}$ with ramification locus $Y$ has degree $N$ if the restriction $\pi$ to $\pi^{-1}\left(\mathbb{D}\setminus Y\right)$ is a covering of degree~$N$. 
\end{definition}

For sake of simplicity, in this paper, we will only consider simple ramified coverings but it is easy to extend the results to general ramified covering by using Chapter $5$ of~\cite{Levy1}.

\begin{definition}
Let $R$ be a ramified covering of the disk with ramification locus $Y$. Let $x \in Y$ be a ramification point of $R$. It is a simple ramification point if there exists $p_0\in \pi^{-1}(x)$ such that ${\sf or} ( p) = 2$, and for any other $p \in \pi^{-1}(x)$, ${\sf or }( p)=1$. The ramified covering $R$ is simple if for any $x \in Y$, $x$ is a simple ramification point. 
\end{definition}

Often we will denote the covering $\pi:R\to \mathbb{D}$ just by $R$. The set of simple ramified coverings of the disk is too big to be interesting. As one does for the theory of random maps, we will only work with the isomorphism classes of simple ramified coverings.

\begin{definition}
Let $\pi: R \to \mathbb{D}$ and $\pi': R' \to \mathbb{D}$ be two simple ramified coverings. They are isomorphic if there exists a homeomorphism $h: R \to R'$ such that $\pi' \circ h = \pi.$
\end{definition}

Let $N$ be a positive integer. We denote by $\mathcal{R}^{N}(Y)$ the set of isomorphism classes of simple ramified coverings of degree $N$ of $\mathbb{D}$ with ramification locus equal to $Y$. In fact, it is even easier to work with labelled simple ramified coverings since the set of automorphisms of a labelled ramified covering is trivial. 

\begin{definition}
Let $\pi: R \to  \mathbb{D}$ be a simple ramified covering of the disk of degree $N$ with ramification locus equal to $Y$. Let $x$ be in $\mathbb{D} \setminus Y$. A labelling $l$ of $R$ at the point $x$ is a bijection from $\{1,...,n\}$ to $\pi^{-1}\left(x\right)$. The pair $(R,l)$ is a labelled simple ramified covering based at $x$. 

Let $(R, l)$, $(R,l')$ be two labelled simple ramified coverings based at $x$. They are isomorphic if there exists an isomorphism of simple ramified coverings $h: R \to R'$ such that $h \circ l = l'$. 
\end{definition}
Let $x$ be a point of $\mathbb{D}\setminus Y$. The set of isomorphism classes of labelled simple ramified coverings of $\mathbb{D}$ with ramification locus $Y$ based at $x$ and with degree $N$ is denoted by $\mathcal{R}_{x}^{N}(Y)$. In order to define a measure on $\mathcal{R}^{N}(Y)$ or $\mathcal{R}_{x}^{N}(Y)$, we need to define a $\sigma$-field. The $\sigma$-field we will consider will be a Borel $\sigma$-field.

\begin{definition}
We consider on $\mathcal{R}^N_x(Y)$ the topology generated by: 
\begin{align*}
\mathcal{V}\left(\left(R,l\right),U\right) = \big\{R' \in \mathcal{R}^N_x(Y)| R_{| M\setminus U} \simeq R'_{| M\setminus U}\big\},
\end{align*}
where $U$ is any open subset such that $Y \subset U \subset \mathbb{D}\setminus x$.

Also, we consider on $\mathcal{R}^N(Y)$ the topology generated by: 
\begin{align*}
\mathcal{V}(R,U) = \big\{ R' \in \mathcal{R}^N(Y)| R_{| M\setminus U} \simeq R'_{| M\setminus U}\big\}, 
\end{align*}
where $U$ is any open subset such that $Y \subset U \subset \mathbb{D}$. 
\end{definition}

Let $\mathcal{T}_{N}$ be the set of transpositions in $\mathfrak{S}(N)$. The set $\mathcal{R}^N_x(Y)$ is in bijection with $\left(\mathcal{T}_{N}\right)^{\#Y}$: this is a finite set, and thus we can consider the uniform measure on $\mathcal{R}^N_x(Y)$. 
When one wants to define a measure on a finite set of objects, it is common to take into account the size of the automorphism group: in case of labelled ramified coverings, the uniform measure is the natural one. 

\begin{definition}
The uniform measure on $\mathcal{R}^N _x(Y)$ is:
\begin{align*}
{\sf \bold{U}}^{N}_{x,Y} = \frac{1}{\left(\#\mathcal{T}_n\right)^{\#Y}}\sum_{(R, l) \in \mathcal{R}^N_{x}(Y)} \delta_{(R, l)}. 
\end{align*}
The natural measure on $\mathcal{R}^{N}(Y)$ is: 
\begin{align*}
{\sf \bold{U}}^{N}_{Y} =  \frac{1}{\left(\#\mathcal{T}_n\right)^{\#Y}}\sum_{R \in \mathcal{R}^{N}(Y)} \frac{n!}{\#{\sf Aut}( R)}\delta_{R}. 
\end{align*}
\end{definition}

Using the Equation $(63)$ of \cite{Levy1}, one gets the following lemma. 
\begin{lemma}
Let $\mathcal{F}: \mathcal{R}^N_x(Y) \to \mathcal{R}^{N}(Y)$ be the application where one forgets about the labelling. We have ${\sf \bold{U}}_{Y}^{N} = {\sf \bold{U}}_{x,Y}^{N} \circ \mathcal{F}^{-1}. $\end{lemma}

 Let $\mathcal{P}_N(dY)$ be a Poisson point process on $\mathbb{D}$ of intensity equal to $\frac{N}{2} dx$. On the set of finite subsets of $\mathbb{D}$, $F(\mathbb{D})$, we will consider the topology which makes the bijection $F(\mathbb{D}) \simeq \cup_{k\geq 0} (\mathbb{D}^{k}\setminus \Delta_{k}) / \mathfrak{S}_k$ continuous. In \cite{Levy1}, Proposition $5.3.3$, T. L\'{e}vy showed that: 

\begin{lemma}
The application which sends $Y$, a finite subset of $\mathbb{D}$, on ${\sf \bold{U}}^{N}_{Y}$ and the one which sends $Y$, a finite subset of $\mathbb{D}\setminus \{x\}$, on ${\sf \bold{U}}_{x,Y}^{N}$ are continuous. 
\end{lemma}

Thus we can define the following measures on simple ramified coverings on the disk (labelled or not). 
\begin{definition}
We consider on $\mathcal{R}^N(Y)$ and $\mathcal{R}^N_x(Y)$ respectively the Borel measures: 
\begin{align*}
{\sf \bold{U}}^N = \int {\sf \bold{U}}_Y^{N} \mathcal{P}_N(dY) \text{ and }
{\sf \bold{U}}_{x}^N  =  \int {\sf \bold{U}}_{x, Y}^{N} \mathcal{P}_N(dY). 
\end{align*}
\end{definition}

The main result in this article is that, in some sense, the measures ${\sf \bold{U}}^N$ or ${\sf \bold{U}}^N_{x}$ converge as $N$ goes to infinity. This assertion has to be taken non-rigorously as the measures are not supported by the same space and the limiting object is not defined. What we will show instead is that the monodromies of the ramified coverings converge in probability. From now on, we will only consider the measure ${\sf \bold{U}}^N_{x}$ on labelled ramified coverings. The case of non labelled ramified coverings could be also studied, yet it would be necessary to be a little more careful on how we define the associated holonomy process thus, for sake of clarity, we prefered to present the results in the setting of labelled ramified coverings.

Let $R$ be a ramified covering in $\mathcal{R}_{x}^{N}(Y)$ and let $l$ be the labelling of the sheets of $R$ at $x$. Let $c$ be a rectifiable loop in $\mathbb{D}$ based at $x$. We can transport the labelling $l$ along the path $c$: it gives us an other labelling $l'$ of the sheets above $x$. The unique element $\sigma  \in \mathfrak{S}(N)$ such that $l'=l \sigma$ is called the monodromy of $R$ along $c$ with respect to $l$ and is denoted by ${\sf mon}_{R,l}( c)$. Suppose that we label $R$ at $x$ with $l \circ \eta$ where $\eta \in \mathfrak{S}(N)$, then $c$ transports the labelling $l \eta$ on $l {\sf mon}_{R,l}(c ) \eta$: it shows that for any curves $c_1$ and $c_2$ based at $x$, 
\begin{align}
\label{diagconj}{\sf mon}_{R,l}(c_1 ) &= \eta^{-1} {\sf mon}_{R,l \circ \eta}( c_1) \eta, \\
{\sf mon}_{R,l}(c_1^{-1}) &=\left( {\sf mon}_{R,l}(c_1 )\right)^{-1}, \\
{\sf mon}_{R,l}(c_1 c_2) &= {\sf mon}_{R,l}(c_2) {\sf mon}_{R,l}(c_1), 
\end{align}
where we recall that $c_1c_2$ is the concatenation of $c_1$ with $c_2$ and $c_1^{-1}$ is the curve $c_1$ with reversed orientation. 

If $c$ is a rectifiable curve, $\mathcal{P}_N(dY)$-a.s. the range of $c$ is inside $\mathbb{D}\setminus Y$. We can thus define the holonomy process associated with ${\sf \bold{U}}^{N}_x$ by using the monodromy along any rectifiable loop based at $x$. The set of rectifiable loops in $\mathbb{D}$ based at $x$ is denoted by $L_{x}(\mathbb{D})$. 

\begin{definition}
The random holonomy field on $L_{x}(\mathbb{D})$ associated with ${\sf \bold{U}}^{N}_x$ is the process $\left( {\sf m}( c) \right)_{c \in L_{x}(\mathbb{D})}$ defined on $\left( \mathcal{R}_{x}^{N}, {\sf \bold{U}}^{N}_x\right)$ where: 
\begin{align*}
 {\sf m}_N(c )\!: \ \ \mathcal{R}_{x}^{N} &\to \mathfrak{S}(N) \\
  (R,l)&\mapsto {\sf mon}_{R,l}(c ). 
\end{align*}
For any $c_1$ and $c_2$ in $L_x(\mathbb{D})$, ${\sf \bold{U}}^{N}_x$-a.s, 
\begin{align*}
{\sf m}_{N}(c_1 c_2) &= {\sf m}_{N}(c_2){\sf m}_{N}(c_1), \\
{\sf m}_{N}\left(c_1^{-1}\right) &= {\sf m}_{N}(c_1)^{-1}.
\end{align*}
\end{definition}

It is quite natural to wonder how a change of the base point $x$ changes the random holonomy field: in order to do so, we need to consider the same index set for the random processes.
\begin{definition}
 Let $c_{0 \to x}$ be a path from $0$ to $x$. The random holonomy field on $L_{0}(\mathbb{D})$ associated with $ {\sf \bold{U}}^{N}_x$ is the process $\left( {\sf m}_{N}\left( c_{0 \to x}^{-1}\ c\ c_{0 \to x}\right)\right)_{c \in L_{0}(\mathbb{D})}$ defined on $\left(\mathcal{R}_{x}^{N}, {\sf \bold{U}}^{N}_x\right)$. 
\end{definition}

Its law does not depend on the choice of $c_{0\to x}$ and it was also proved by L\'{e}vy that the laws of the random holonomy field on $L_{0}(\mathbb{D})$ associated with $ {\sf \bold{U}}^{N}_x$ do not depend on the choice of $x$. From now on, we will only consider the random holonomy field on $L_{0}(\mathbb{D})$ associated with ${\sf \bold{U}}^{N}_0$. Let us state a theorem which is a direct consequence of Proposition $5.4.4$ of \cite{Levy1}. 

\begin{theorem}
\label{th:egaliYMRam}
Let $\left(S_t^{N}\right)_{t \geq 0}$ be a $\mathcal{T}_N$-random walk on $\mathfrak{S}(N)$. The holonomy field on $L_0(\mathbb{D})$ associated with ${\sf \bold{U}}^{N}_0$ has the same law as the process $\left(h(l)\right)_{l \in L_0(\mathbb{D})}$ under the $\mathfrak{S}(N)$-valued Yang-Mills measure associated with $\left(S_t^{N}\right)_{t \geq 0}$. 
\end{theorem}

In a nutshell we have the following ``equality'': 
\begin{align*}
\sf{ Monodromy\ of\ random\ ramified\ coverings = \mathfrak{S}(N)\ valued\ Yang\!-\!Mills\ measure. }
\end{align*}

Using Theorems \ref{masterfield} and  \ref{th:egaliYMRam}, we have proved in this article that the traces of the monodromies of random ramified coverings of the disk of degree $N$ converge in probability as $N$ goes to infinity.

\begin{theorem}
There exists an non-random application:  
\begin{align*}
\phi: L_{0}(\mathbb{D}) &\to \mathbb{R} \\
l &\mapsto \phi(l), 
\end{align*}
which is continuous for the convergence with fixed endpoints such that for any loop $l \in L_0(\mathbb{D})$, $\frac{1}{N}Tr\left({\sf m}_N(l)\right)$ converges in probability to $\phi(l)$ as $N$ goes to infinity. The application $\phi$ is the $\mathfrak{S}(\infty)$-master field associated with the transpositions. 
\end{theorem}

\section{Computation of the $\mathfrak{S}(\infty)$-master field}
\label{sec:computation}

Let us remark that the proofs of the existence of the $\mathfrak{S}(\infty)$-master field allows us to compute explicitely the $\mathfrak{S}(\infty)$-master field on loops which can be drawn on graphs. Indeed, given such a loop $l_0$, there exists a graph $\mathbb{G}$ and facial lassos $\left({\sf l}_{F} \right)_{F \in \mathbb{F}}$ defined on this graph such that the random variable $h_{l_0}$ is a product of random variables of the form $h_{{\sf l}_{F}}$ or $h_{{\sf l}_{F}}^{-1}$, with $F\in\mathbb{F}$. Yet, the random variables $(h_{{\sf l}_F})_{F \in \mathbb{F}}$ are independent and for any $F\in \mathbb{F}$, $h_{{\sf l}_F}$ has the same law as $S^{N}_{dx(F)}$. When $N$ goes to infinity, the family $\left(h_{{\sf l}_F} \right)_{F \in \mathbb{F}}$ converges in $\mathcal{P}$-distribution and $h_{{\sf l}_{F_0}}$ is asymptotically $\mathcal{P}$-free with $\left(h_{{\sf l}_F} \right)_{F \in \mathbb{F}\setminus \{F_0\}}$ for any bounded face $F_0$. This $\mathcal{P}$-freeness allows us to compute the $\mathfrak{S}(\infty)$-master field. 

Let us give an exemple of a computation of $\phi(l)$, where $\phi$ is the $\mathfrak{S}(\infty)$-master field. Let us consider the loop $l$ as drawn in Figure \ref{fig:2}. Let us consider the loops $a$, $b$ and $c$ as drawn in Figure \ref{fig:3}. 

\begin{figure}[h!]
\centering
\resizebox{0.75\textwidth}{!}{
  \includegraphics[width=2pt]{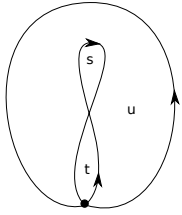}
}
\caption{The loop $l$.}
\label{fig:2}       
\end{figure}
\begin{figure}[h!]
\centering
\resizebox{0.75\textwidth}{!}{
  \includegraphics[width=9.5pt]{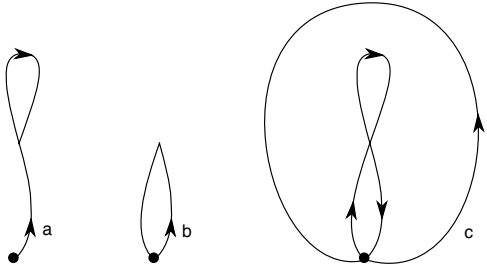}
}
\caption{The decomposition in facial lassos.}
\label{fig:3}       
\end{figure}

We have the decomposition: 
\begin{align*}
l = aba^{-1}bc. 
\end{align*}
Thus, we have the decomposition for the holonomy field: 
\begin{align*}
h(l) = h(c) h(b) h(a^{-1}) h(b) h(a). 
\end{align*}
For any $N$, under the Yang-Mills measure with $\mathfrak{S}(N)$ structure group, the random variables $h(a)$, $h(b)$ and $h(c)$ are independent and they have respectively the law of $S_{s}^{N}$, $S_{t}^{N}$ and $S_{u}^{N}$, where $(S_t^{N})_{t \geq 0}$ is the walk by transpositions on $\mathfrak{S}(N)$. For any integer $N$, let $A_N$, $B_N$, $C_N$ be three independent variables which have respectively  the law of $S_{s}^{N}$, $S_{t}^{N}$ and $S_{u}^{N}$. We want to compute: 
\begin{align*}
\phi(l)=\lim_{N \to \infty} \frac{1}{N} \mathbb{E}[Tr(C_NB_NA_N^{-1}B_NA_N)]. 
\end{align*}
Let us recall a consequence of Theorem 3.4 of \cite{Gab2}. 
\begin{lemma}
\label{lemma:calcul}
Let $M$ and $L$ be two elements of $L^{\infty^{-}} \otimes \mathcal{M}(\mathbb{C})$ which converge in $\mathcal{P}$-distribution and which are asymptotically $\mathcal{P}$-free. Then: 
\begin{align*}
\mathbb{E}m_{id_1}[ML] &= \mathbb{E}m_{id_1}[M]\mathbb{E}m_{id_1}[L]\\
\mathbb{E}m_{(1,2)}[ML,ML^{-1}]& = \mathbb{E}\kappa_{id_{2}}[M] + \mathbb{E}\kappa_{{\sf 0}_2}[M] \mathbb{E}m_{{\sf 0}_2}[L,L^{-1}] \\&\ \ \ \ \ \ \ \ \ \ \ \ \ \ \ + \mathbb{E}\kappa_{(1,2)}[M] \mathbb{E}m_{id_2}[L,L^{-1}], 
\end{align*}
where $(1,2)$ is the transposition and ${\sf 0}_2 = \{1,2,1',2'\}$. 
\end{lemma}

Using this lemma, we get that $\lim_{N \to \infty} \frac{1}{N} \mathbb{E}[Tr(C_NB_NA_N^{-1}B_NA_N)]$ is equal to: 
\begin{align*}
\left(\lim_{N \to \infty} \frac{1}{N} \mathbb{E}[Tr(C_N)]\right) \left(\lim_{N \to \infty} \frac{1}{N} \mathbb{E}[Tr(B_NA_N^{-1}B_NA_N)]\right). 
\end{align*}
We already computed $\left(\lim_{N \to \infty} \frac{1}{N} \mathbb{E}[Tr(C_N)]\right)$ in the proof of Lemma \ref{estime}: it is equal to $e^{-u}$. Besides: $$ \left(\lim_{N \to \infty} \frac{1}{N} \mathbb{E}[Tr(B_NA_N^{-1}B_NA_N)]\right) = \mathbb{E}m_{(1,2)}[BA,BA^{-1}].$$ After some easy computations similar to the one we did in Lemma \ref{calculkappa}, using the results given in Theorem \ref{permutations} and in Lemma \ref{lemma:calcul}, we get that: 
\begin{align*}
\phi(l) = e^{-u}\left[ e^{-2t} + e^{-t}e^{-s} - e^{-2t}e^{-s} + e^{-2t}e^{-2s} - e^{-2t}e^{-2s} t \right].\\
\end{align*}

{\em Acknowledgements.} The author would like to gratefully thank A. Dahlqvist and G. C\'{e}bron for the useful discussions, his PhD advisor Pr. T. L\'{e}vy for his helpful comments and his postdoctoral supervisor, Pr. M. Hairer, for giving him the time to finalize this article. 

The first version of this work has been made during the PhD of the author at the university Paris 6 UPMC. This final version of the paper was completed during his postdoctoral position at the University of Warwick where the author is supported by the ERC grant, “Behaviour near criticality”, held by Pr. M. Hairer.

\end{document}